\documentclass[reqno,10pt]{amsart}
\usepackage{amscd,amssymb,amsmath,color}
\usepackage{paralist}
\usepackage{latexsym}
\usepackage[all]{xy}
\usepackage[colorlinks=true]{hyperref}

\def\rk{{\rm rk}}
\def\tiltau{{\widetilde{\tau}}}
\def\bfF{{\bf F}}
\def\bfN{{\bf N}}
\def\bfQ{{\bf Q}}
\def\bfZ{{\bf Z}}
\def\bfR{{\bf R}}
\def\bfA{{\bf A}}
\def\calO{{\mathcal O}}
\def\calM{{\mathcal M}}
\def\gp{{\rm gp}}
\def\ol{\overline}
\def\ot{{\ol{t}}}
\def\ov{{\ol{v}}}
\def\oM{{\ol{M}}}
\def\oK{{\ol{K}}}
\def\oX{{\ol{X}}}
\def\oR{{\ol{R}}}
\def\ocalO{{\ol\calO}}

\def\orho{{\ol{\rho}}}
\def\osigma{{\ol{\sigma}}}
\def\otau{{\ol{\tau}}}

\def\oZ{{\ol{Z}}}
\def\oQ{{\ol{Q}}}
\def\oS{{\ol{S}}}
\def\bfP{{\bf P}}
\def\oY{{\ol{Y}}}
\def\of{{\ol{f}}}
\def\oh{{\ol{h}}}

\def\opsi{{\ol{\psi}}}
\def\ophi{{\ol{\phi}}}
\def\oT{{\ol{T}}}
\def\oP{{\ol{P}}}
\def\oL{{\ol{L}}}

\def\sat{{\rm sat}}

\def\..{{,\dots,}}
\def\:{{\colon}}
\def\GL{{\rm GL}}

\def\int{{\rm int}}
\def\sat{{\rm sat}}

\def\Mor{{\rm Mor}}
\def\Spec{{\rm Spec}}
\def\oSpec{{\ol{\rm Spec}}}
\def\toisom{{\xrightarrow\sim}}
\def\onto{\twoheadrightarrow}
\def\into{\hookrightarrow}

\newtheorem{theor}[subsubsection]{Theorem}
\newtheorem{prop}[subsubsection]{Proposition}
\newtheorem{lem}[subsubsection]{Lemma}
\newtheorem{cor}[subsubsection]{Corollary}

\newtheorem{conjecture}[subsubsection]{Conjecture}

\def\conv{{\rm conv}}
\def\Cay{{\mathcal{C}}}
\def\st{{\rm st}}

\theoremstyle{definition}

\newtheorem{rem}[subsubsection]{Remark}

\newtheorem{exam}[subsubsection]{Example}

\begin{document}

\author{Karim Adiprasito}
\author{Gaku Liu}
\author{Igor Pak}
\author{Michael Temkin}

\date{\today}

\title{Log smoothness and polystability over valuation rings}

\address{Einstein Institute of Mathematics, The Hebrew University of Jerusalem, Giv'at Ram, Jerusalem, 91904, Israel}
\email{adiprasito@math.huji.ac.il}

\address{Max Planck Institute for Mathematics in the Sciences, 04103 Leipzig, Germany}
\email{gakuliu@gmail.com}

\address{Department of Mathematics, UCLA, 90095
Los Angeles, USA}
\email{pak@math.ucla.edu}

\address{Einstein Institute of Mathematics, The Hebrew University of Jerusalem, Giv'at Ram, Jerusalem, 91904, Israel}
\email{michael.temkin@mail.huji.ac.il}

\thanks{K.A. was supported by ERC StG 716424 - CASe and ISF Grant 1050/16. G.L. was supported by NSF grant 1440140. I.P. was supported by NSF grant 1700444.
M.T. was supported by ERC Consolidator Grant 770922 - BirNonArchGeom and ISF Grant 1159/15.
}
\begin{abstract}
Let $\calO$ be a valuation ring of height one of residual characteristic exponent $p$ and with algebraically closed field of fractions. Our main result provides a best possible resolution of the monoidal structure $M_X$ of a log variety $X$ over $\calO$ with a vertical log structure: there exists a log modification $Y\to X$ such that the monoidal structure of $Y$ is polystable. In particular, if $X$ is log smooth over $\calO$, then $Y$ is polystable with a smooth generic fiber. As a corollary we deduce that any variety over $\calO$ possesses a polystable alteration of degreee $p^n$. The core of our proof is a subdivision result for polyhedral complexes satisfying certain rationality conditions.
\end{abstract}

\maketitle

\section{Introduction}

\subsection{Motivation}
The original motivation for undertaking this project was the following question: can any log smooth scheme (or formal scheme) over a valuation ring $\calO$ of height one be modified to a polystable one? It seems that the expert community did not have a clear consensus if the answer should be yes or no, though at least few experts hoped the answer might be positive. Our main result proves that this is indeed the case. Moreover, we develop the theory in the more general form of ``resolving the monoidal structure`` of an arbitrary log variety over $\calO$ with a vertical log structure (by which we mean a log scheme over $\calO$, flat and of finite presentation, and such that the log structure is trivial on the generic fiber, see \S\ref{logvarsec}). In particular, our main result extends the absolute monoidal resolution of fine log schemes of Gabber, see \cite[Theorem~3.3.16]{Illusie-Temkin}, to this relative (and non-fine) case.

\subsection{Background}
To explain what the above question means, why it is important, and where it comes from we shall recall some history.

\subsubsection{Semistable modification}
To large extent, toroidal geometry and related combinatorial methods in algebraic geometry originate from the seminal work \cite{KKMS}. Its aim was to prove the semistable reduction (or modification) theorem over a discrete valuation ring $\calO$ of residual characteristic zero: if $X$ is an $\calO$-variety with smooth generic fiber $X_\eta$ then there exists a modification $X'\to X$ such that $X'_\eta=X_\eta$ and $X'$ is semistable. Recall, that the latter means that \'etale-locally $X'$ admits an \'etale morphism to standard semistable varieties of the form $$\Spec(\calO[t_0\..t_n]/(t_0\dots t_l-\pi)),$$ where $\pi\in\calO$ (sometimes one also wants $\pi$ to be a uniformizer). The proof goes in two steps: (1) using Hironaka's theorem one reduces to the case when $X$ is toroidal over $\calO$, (2) for toroidal $\calO$-varieties the question is translated to geometry of polytopes with rational vertices and is then resolved by combinatorial methods. Certainly, part (1) is much deeper, but, once Hironaka's theorem is available, the reduction is easy. So, \cite{KKMS} is mainly occupied with constructing the translation (that is, building foundations of toroidal geometry, see \cite[Chapters~I and II]{KKMS}), and solving the (still rather difficult) combinatorial problem about subdividing polytopes with rational vertices into simplices of especially simple form, see \cite[Chapter~III]{KKMS}.

\subsubsection{Log geometry}
This work is written within the framework of log geometry, which nowadays gradually subsumes toroidal geometry. In this setting the above strategy receives the following natural interpretation: first one resolves $X$ obtaining a log smooth variety $Y$ over $\calO$, and then one improves the monoidal structure of $Y$ by log smooth modifications only (e.g. log blow ups). In particular, we will work with log structures which are finitely generated over the log structure of $\calO$, but are not fine.

\subsubsection{Classical combinatorics}\label{classcombsec}
In toric and toroidal geometries, combinatorics is used to model various properties of certain varieties. In a nutshell, on the side of functions one considers rational polyhedral cones $\sigma$ in a vector space $M$ with a fixed lattice $\Lambda$, and on the geometric side one considers dual cones $\sigma^*$ in $N=M^*$ and fans obtained by gluing such cones. If one works over a discrete valuation ring $\calO$, then any uniformizer $\pi$ induces a canonical element $v_\pi\in\Lambda$ and one can also restrict to the hyperplane $N_1$ given by $v_\pi=1$. In this way, one passes from cones to polytopes. Finally, one can also use the equivalent language of toric monoids $P$, their spectra $\Spec(P)$ and fans obtained by gluing the latter. Then $\sigma$ and $\sigma^*$ correspond to $P=\sigma\cap\Lambda$ and $\Spec(P)$. The latter language is essentially due to Kato, and it is convenient for applications to log geometry of fine schemes.

\subsubsection{Limitations of semistability}
Assume now that $\calO$ is an arbitrary valuation ring of height one. In addition, we assume that its fraction field $K$ is algebraically closed. (Since semistable reduction is studied up to finite extensions of $K$, this does not really restrict the generality.) In particular, we can view the group of values of $K$ as a $\bfQ$-vector subspace $V$ of $\bfR$. Semistability over $\calO$ can be studied in two similar contexts: (1) for schemes of finite type over $\calO$, (2) for formal schemes over $\calO$, if the latter is complete. Let $X$ be a (formal) $\calO$-scheme of finite type with smooth generic fiber $X_\eta$. It was known to experts for a long time that semistable modification of $X$ might not exist when $\dim(X_\eta)>1$ and $\dim_\bfQ(V)>1$ (for example, see \cite[Remark~3.4.2]{temst}).

Hear is a quick explanation. Combinatorics of log smooth $\calO$-schemes can be encoded by $(\bfZ,V)$-polytopes, or simply $V$-polytopes, introduced by Berkovich in \cite{bercontr} (loc.cit. uses exponential notation). These are polytopes in $\bfR^n$ whose vertices are in $V^n$ and edge slopes are rational. The main combinatorial result of \cite{KKMS} describes triangulations of such polytopes when $\dim_\bfQ(V)=1$. However, if $\dim_\bfQ(V)>1$ then no triangulation might exist even in dimension 2, see also Section~\ref{ssec:semis} for a specific example.

\subsubsection{Polystable reduction conjecture}
Polystable (formal) schemes over $\calO$ were introduced by Berkovich in \cite[Section 1]{bercontr}; \'etale locally they are products of semistable ones. The above example shows that over a general base the best one can hope for is a polystable modification. In fact, it was conjectured in \cite[Conjecture~0.2]{AK} (the terminology is different!) and \cite[Remark~3.4.2]{temst} that polystable modification does exist. Over $\calO$ the conjecture reads as follows:

\begin{conjecture}\label{polyconj}
Let $\calO$ be as above and assume that $X$ is a flat $\calO$-scheme of finite type whose generic fiber $X_\eta$ is smooth. Then there exists a modification $Y\to X$ such that $Y_\eta=X_\eta$ and $Y$ is polystable over $\calO$.
\end{conjecture}

\subsection{Main results}

\subsubsection{The monoidal resolution theorem}
As in the case of a DVR, it is natural to attack Conjecture~\ref{polyconj} in two steps: (1) find a modification $Y\to X$ with a log smooth $Y$, (2) make the log structure of $Y$ polystable using log modifications. Similarly to \cite{KKMS} in the case of DVR's, this paper establishes step (2), which is the combinatorial part of the problem. In fact, Theorem~\ref{monoidalth} even proves that the log structure of an arbitrary log variety over $\calO$ can be made polystable by a log modification.

\begin{rem}
\begin{compactenum}[(i)]
\item Our result provides a good evidence to expect that polystable modification always exists, at least, there are no combinatorial obstacles and the question reduces to establishing step (1). See, also \S\ref{appsec} below.

\item In addition, our result highlights importance of the class of polystable formal schemes for Berkovich geometry.
\end{compactenum}
\end{rem}

\subsubsection{Application to polystable modification/alteration}\label{appsec}
In general, question (1) is very hard, at the very least, not easier than resolution of singularities and semistable reduction in positive characteristic. However, if the residual characteristic is zero then an affirmative answer immediately follows from the semistable reduction of Abramovich-Karu \cite[Theorem~2.1]{AK}, and in general this can be achieved by the $p$-alteration theorem \cite[1.2.9]{tame-distillation} (which improves on results of de Jong, Gabber and Illusie-Temkin). As a corollary we will prove in Theorem~\ref{appth} that any log variety $X$ over $\calO$ possesses an alteration $X'\to X$ of degree $p^n$, where $p$ the residual characteristic exponent of $\calO$, such that $X'$ is polystable over $\calO$. In particular, in the equal characteristic zero case, polystable modification does exists.

\begin{rem}
\begin{compactenum}[(i)]
\item Even in the equal characteristic zero case, the weak point of the cited results and our Theorem~\ref{appth} is that the modification $X'\to X$ modifies the generic fiber $X_\eta$ even when it is smooth. In particular, it does not provide polystable models of varieties over $K$, whose existence still remains a (now very plausible) conjecture. However, it seems very probable that this finer problem will be solved soon by a new logarithmic desingularization algorithm developed in \cite{ATW} (so far, for log varieties over a field). This would settle Conjecture~\ref{polyconj} in the equal characteristic zero case.

\item
We shall also mention that in \cite[Theorem~1.7]{Forre} Patrick Forr\'e claims a strong form of polystable reduction over valuation rings of height $h$ with group of values $\bfN^h$, although we were not able to verify all details
\end{compactenum}
\end{rem}

\subsubsection{Combinatorics of log varieties over $\calO$}\label{combOsec}
As in the classical case, there are several equivalent ways to describe combinatorics of toric and toroidal geometries over $\calO$. In fact, we extend all classical approaches outlined in \S\ref{classcombsec} to this case, and prove their equivalence. On the side of functions, we introduce categories of $R$-toric monoids, where $R=V\cap\bfR_{\ge 0}$, and $V$-special cones, and establish their equivalence in Theorem~\ref{conemonoidth}. On the geometric side, we introduce the categories of $R$-toric fans, $V^\bot$-special cone complexes, and $V$-affine polyhedral complexes, establish their equivalence in Theorem~\ref{nonembeqth}, and prove that the latter preserves various geometric properties, such as simpliciality and polystability, see \S\ref{equivsec}.

\begin{rem}
\begin{compactenum}[(i)]
\item We decided to work out foundations in such detail because each approach has advantages of its own. In particular, fans are most naturally related to monoidal structure of log varieties over $\calO$, the language of $V$-affine polytopes is most convenient to prove the subdivision theorem, and the two theories are naturally compared through $V$-special and $V^\bot$-special cones. Also, we expect that these foundations will be useful for further research of (formal) varieties over $\calO$.

\item Gubler-Soto used $V^\bot$-special complexes in \cite[Section~4]{Gubler-Soto} to study toric $\calO$-varieties. In loc.cit. they were called
 ``$V$-admissible fans''.

\item The language of $V$-affine polytopes (and its analogue for PL spaces) was used by Berkovich to describe a natural structure on skeletons of analytic spaces over $\calO$.
\end{compactenum}
\end{rem}

\subsubsection{Future directions}
In \cite{AK} Abramovich and Karu conjectured that any morphism $P\to Q$ of rational polyhedral complexes can be made a polystable one by an appropriate subdivision of $P$ and $Q$ and altering the lattice of the base, see \cite[Section 8]{AK} for details. In fact, our main combinatorial result, Theorem~\ref{compsubdv}, can be viewed as a local on the base version of that conjecture, and the two problems are tightly related. In particular, Dan Abramovich pointed out that the first stage of our construction, existence of a polysimplicial subdivision, was established in the global setting in \cite{AR}. The conjecture of Abramovich-Karu can be treated by our method as well, and is the subject of \cite{ALT}.

\subsection{Structure of the paper}
Finally, let us give an overview of the paper. In Section~\ref{monsec} we recall basic facts about monoids and their spectra. The results are rather standard but hard to find in the literature, especially in the generality of monoids which are not finitely generated. In \S\ref{toricmonsec} we introduce $R$-toric monoids over a sharp valuative monoid $R$ of height one and $R$-toric fans obtained by gluing spectra of $R$-toric monoids. This extends Kato's approach to combinatorics of fine log schemes to log varieties over $\calO$.

In Section~\ref{fansec} we study the related cone and polyhedral complexes, and prove the equivalence theorems mentioned in \S\ref{combOsec}. This extends the combinatorics of the classical toroidal/toric geometry to the case of $\calO$-varieties.

Section~\ref{combsec} is devoted to the proof of our main combinatorial result, Theorem~\ref{compsubdv} on polystable subdivisions of $V$-affine polytopes, and then extends it to polyhedral complexes.

Finally, in Section~\ref{lastsec} we apply the results of \S\S\ref{monsec}--\ref{combsec} to study log varieties over $\calO$. As a preparation, we prove in Theorem~\ref{zarth} that any log variety $X$ possesses a log modification $Y\to X$ such that the log structure of $Y$ is Zariski and $Y$ possesses a global fan $F$. Then, we show that a polystable subdivison of $F$ constructed in the earlier sections induces the desired polystable resolution of the log structure of $X$.

\subsection*{Acknowledgments}
We would like to thank D. Abramovich, W. Gubler and M. Ulirsch for pointing out various inaccuracies and suggesting to apply our method to altered resolution of log varieties over $\calO$. We followed this suggestion by proving Theorem~\ref{appth}.

\section{Monoschemes and fans over valuative monoids}\label{monsec}
In this section we collect needed material on integral monoids and their spectra. Some results are perhaps well known to experts, but miss in the literature.

\subsection{Monoids}

\subsubsection{Conventiones}
A {\em monoid} in this paper means a commutative monoid $M=(M,+)$. As usually, $M^\times$ is the subgroup of units, and $\oM=M/M^\times$ is the sharpening of $M$. We will use notation $M\oplus\bfN^n=M[t_1\..t_n]=M[t]$, where $t=(t_1\..t_n)$ is the basis of $\bfN^n$.

Recall that $M^\int$ is the image of $M$ in its Grothendieck group $M^\gp$, and $M$ is integral if $M=M^\int$. Although we will mainly work within the category of integral monoids, we will use full notation to avoid confusions. For example, an integral pushout will be denoted $(M_1\oplus_MM_2)^\int$.

\subsubsection{Splittings}
Studying of toric monoids easily reduces to the sharp case because units can be split off. In general, one should be more careful.

\begin{lem}\label{unitmonlem}
An integral monoid $P$ splits as $P=P^\times\oplus\oP$ if and only if the associated short exact sequence of abelian groups $P^\times\to P^\gp\to\oP^\gp$ splits.
\end{lem}
\begin{proof}
The direct implication is clear. Conversely, if $\phi\:\oP^\gp\into P^\gp$ splits the sequence, then $P=P^\times\oplus\phi(\oP)$ since $P$ is the preimage of $\oP$ under $P^\gp\onto\oP^\gp$.
\end{proof}

\begin{rem}
If $P$ is saturated then $\oP^\gp$ is torsion-free. In general, this is not enough for splitting, but in the toric case $\oP^\gp$ is finitely generated, and hence free.
\end{rem}

Another useful way to simplify a monoid $P$ is to split off a free part. We say that an element $t\in P$ is {\em prime} if the ideal $(t)=P+t$ is prime. The following lemma is obvious.

\begin{lem}\label{freeofflem}
An element $t\in P$ is prime if and only if $P$ splits as $P=Q[t]$. In addition, $Q=P\setminus (t)$, so the splitting is uniquely determined by $t$.
\end{lem}

\subsubsection{Finite presentation}
An $M$-monoid $N$ is a monoid provided with a homomorphism $\phi\:M\to N$. We say that $N$ is {\em finitely generated} over $M$ if it is generated over $\phi(N)$ by finitely many elements. In other words, $\phi$ can be extended to a surjective homomorphism $\psi\:M[t_1\..t_n]\onto N$. If, in addition, one can choose $\psi$ defined by an equivalence relation generated by finitely many relations $f_i=g_i$ with $f_i,g_i\in M[t]$, then we say that $N$ is of {\em finite presentation} over $M$.

\begin{lem}\label{fplem}
Assume that $\phi\:M\to N$ is an injective homomorphism of integral monoids, and assume that $N$ is finitely generated over $M$. Then $N$ is finitely presented over $M$ as an integral monoid, in the sense that there exists a finitely presented $M$-monoid $L$ such that $L^\int=N$.
\end{lem}
\begin{proof}
Choose a presentation $N=M[t_1\..t_n]/R$, where $R$ is generated by relations $f_i=g_i$ for $i\in I$. Let $R^\gp$ be the subgroup of $M[t]^\gp=M^\gp\oplus\bfZ^n$ generated by the elements $f_i-g_i$. Since $\phi$ is injective, $R^\gp\cap M^\gp=0$ and hence there exists a finite subset $I'\subseteq I$ such that $R^\gp$ is generated by the elements $f_i-g_i$ with $i\in I'$. Let $R'$ be the equivalence relation generated by equalities $f_i=g_i$ with $i\in I'$. Then $L=M[t_1\..t_n]/R'$ admits a surjective map $L\onto N$ inducing an isomorphism $L^\gp=N^\gp$. In particular, $L^\int=N$.
\end{proof}

\subsubsection{Approximation}
A finitely presented $M$-monoid $N$ is defined already over a fine submonoid of $M$. This is a classical result in the case of rings, and the case of monoids is similar:

\begin{lem}\label{approxlem}
Let $\phi\:M\to N$ be as in Lemma~\ref{fplem} and let $\{M_i\}_{i\in I}$ be the family of fine submonoids of $M$. Then there exist $i\in I$ and a finitely generated integral $M_i$-monoid $N_i$ with an isomorphism $(N_i\oplus_{M_i}M)^\int=N$. Moreover, if $N'_i$ is another such $M_i$-monoid then there exists $j\ge i$ and an isomorphism $(N_i\oplus_{M_i}M_j)^\int=(N'_i\oplus_{M_i}M_j)^\int$ compatible with $(N_i\oplus_{M_i}M)^\int=N=(N'_i\oplus_{M_i}M)^\int$.
\end{lem}
\begin{proof}
By Lemma~\ref{fplem} we have that $N=(M[t_1\..t_n]/R)^\int$, where $R$ is generated by equalities $f_j=g_j$ with $1\le j\le n$ and $f_j,g_j\in M[t]$. The elements $f_j,g_j$ lie in $M_i[t]$ for a large enough $i$. The equalities $f_j=g_j$ generate an equivalence relation $R_i$ on $M_i[t]$, and $N_i=(M_i[t]/R_i)^\int$ is as required. If $N'_i=(M_i[x_1\..x_{n'}]/R'_i)^\int$ is another such $R_i$-monoid, then we have equalities $t_k=m'_k+\sum_i a_{kj}x_j$ and $x_j=m''_j+\sum_i b_{jk}t_k$ in $N$ with $m'_k,m''_j\in M$. Therefore, we can take any $M_j$ containing the elements $m'_1\..m'_n,m''_1\..m''_{n'}$.
\end{proof}

\subsubsection{Valuative monoids}
By a {\em valuative monoid} we mean an integral monoid $M$ such that $M^\gp=M\cup(-M)$. Let us list a few basic facts. Justification is analogous to the parallel theory of valuation rings, so we only indicate it. Setting $m_1\ge m_2$ if $m_1-m_2\in M$ defines an order on $M^\gp$, and $m_1\sim m_2$ if and only if $m_1-m_2\in M^\times$. In particular, $\oM^\gp=M^\gp/M^\times$ acquires a total order and the homomorphism $\nu\:M\to\oM^\gp$ can be viewed as a valuation. In fact, an integral monoid $M$ is a valuative monoid if and only if $\oM$ is of the form $A_{\ge 0}$, where $A$ is a totally ordered abelian group and $A_{\ge 0}=\{a\in A|a\ge 0\}$. Ideals of a valuative monoid $M$ are totally ordered by inclusion, and prime ideals correspond to convex subgroups $H\subseteq\oM^\gp$ via $H\mapsto M\setminus\nu^{-1}(H)$. A valuative monoid is of height $n$ if it has exactly $n$ non-empty prime ideal, and this happens if and only if $\oM^\gp$ is of height $n$. Equivalently, $\oM^\gp$ admits an ordered embedding into $\bfR^l$ if and only if $l\ge n$.

\subsection{Fans, monoschemes, and subdivisions}

\subsubsection{Conventiones}
We will use the same definition as in \cite[Section~2.1]{Illusie-Temkin}, but applied to arbitrary integral monoids, which do not have to be fine. So, throughout this paper a {\em monoscheme} is a monoidal space $X=(X,M_X)$ locally of the form $\Spec(M)$ with an integral monoid $M$. A {\em fan} is a monoidal space which is locally isomorphic to the sharpening $\oX=(X,\oM_X)$ of a monoscheme. We will write $\oSpec(M)$ instead of $\ol{\Spec(M)}$, it depends only on $\oM$.

\begin{rem}
\begin{compactenum}[(i)]
\item Clearly, $\oX$ and the morphism of monoidal spaces $\oX\to X$ depend on $X$ functorially. In some aspects, the sharpening morphism $\oX\to X$ can be viewed as an analog of the reduction morphisms $\mathrm{Red}(X)\into X$ in the theory of schemes. In particular, both only replace the structure sheaf by its quotient.

\item In the theory of schemes, some types of morphisms to $\mathrm{Red}(X)$ can be lifted canonically to morphisms to $X$. For example, this is the case for the class of \'etale morphisms. In the theory of fans and monoschemes, we will now introduce a class of birational morphisms, which possesses a similar lifting property from sharpenings.
\end{compactenum}
\end{rem}

\subsubsection{Morphisms of finite type}
A morphism $f\:Y\to X$ of monoschemes or fans is called {\em locally of finite type} if locally it is of the form $\Spec(N)\to\Spec(M)$ or $\oSpec(N)\to\oSpec(M)$ with $N$ finitely generated over $M$. We say that $f$ is {\em of finite type} if it is quasi-compact and locally of finite type.

\begin{rem}
\begin{compactenum}[(i)]
\item The definition we gave is local. We will not need this, but it is easy to see that in the affine case it suffices to check the global sections: $\Spec(N)\to\Spec(M)$ is of finite type if and only if $N$ is finitely generated over $M$.

\item In view of Lemma~\ref{fplem}, in the category of (integral) monoschemes there is no need to introduce a finer class of finitely presented morphisms.
\end{compactenum}
\end{rem}

\subsubsection{Birational morphisms of monoschemes}
Let $f\:Y\to X$ be a morphism of monoschemes. We say that $f$ is {\em birational} if it induces a bijection between the sets of generic points and $f^\#_y\:M_{X,f(y)}^\gp\toisom M_{Y,y}^\gp$ for any point $y\in Y$.

\begin{rem}
Since $M_X^\gp$ and $M_Y^\gp$ are locally constant, it suffices to check that $f^\#_\eta$ are isomorphisms only for generic points $\eta\in Y$. Since $M_{Y,\eta}^\gp=M_{Y,\eta}$, we see that $f$ is birational if and only if it induces an isomorphism of the monoschemes of generic points of $Y$ and $X$.
\end{rem}

\subsubsection{Birational morphisms of fans}
Passing to fans preserves only the surjectivity of the maps $f^\#_y$, so we say that a morphism $f\:Y\to X$ of fans is {\em birational} if the maps $f^\#_y$ are surjective for any $y\in Y$.

\begin{rem}
In the case of fans, one has to check the surjectivity at all closed points of $Y$, while it is vacuously true for generic points $\eta$ because $\oM_\eta=0$.
\end{rem}

\subsubsection{The lifting property}
We will now prove that birational morphisms of fans lift to monoschemes uniquely, and characterise this lifting by a universal property.

\begin{lem}\label{birlem}
Let $X$ be a monoscheme with associated fan $\oX$ and let $g\:Z\to\oX$ be a birational morphism of fans, then

\begin{compactenum}[(i)]
\item There exists a birational morphism $f\:Y\to X$ such that $\of=g$, and it is unique up to unique isomorphism. In addition, $f$ is of (locally) finite type if and only if $g$ is of (locally) finite type.

\item For any morphism of monoschemes $T\to X$, the natural map $\Mor_X(T,Y)\to\Mor_\oX(\oT,\oY)$ is bijective.
\end{compactenum}
\end{lem}
\begin{proof}
We should construct $f$ and $h$ in the right diagram below and prove that they are unique. It suffices to do this locally on $X$, $Z$ and $T$, since the local liftings then glue by the uniqueness. Also, the condition on being of finite type can be checked locally because the sharpening morphisms are homeomorphisms.

Thus, assuming that $X=\Spec(P)$ and $T=\Spec(L)$ are affine, and $g$ corresponds to a homomorphism $\lambda\:\oP\to R$ such that $\oP^\gp\onto R^\gp$, we can switch from the left diagram below to the right one.
$$
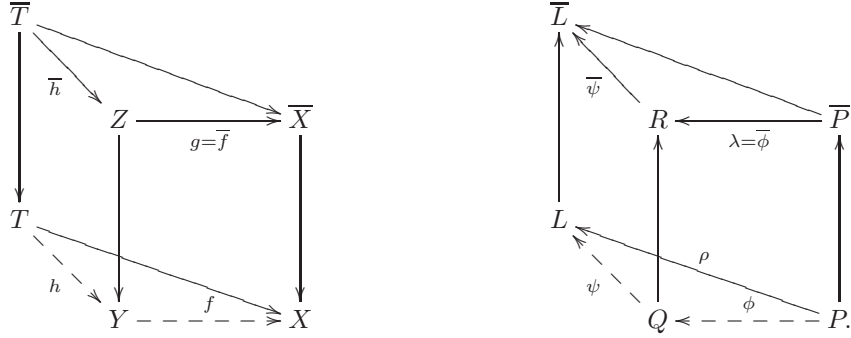
\xymatrix{
\oT\ar[dd]\ar[dr]_\oh\ar[drrr]&&& &&&\oL&&&\\
& Z\ar[rr]_{g=\of}\ar[dd] &&\oX\ar[dd]&&&& R\ar[ul]^\opsi&&\oP\ar[ll]^{\lambda=\ophi}\ar[ulll]\\
T\ar@{-->}[dr]_h\ar[drrr]&&& &&& L\ar[uu]\\
& Y\ar@{-->}[rr]^{f} && X&&& & Q\ar[uu]\ar@{-->}[ul]^\psi && P\ar@{-->}_\phi[ll]\ar[uu]\ar[ulll]_\rho.
}
$$
The induced map $\varphi\:P^\gp\to R^\gp$ is surjective, hence setting $Q=\varphi^{-1}(R)$ we obtain a homomorphism $\phi\:P\to Q$ such that $P^\gp=Q^\gp$ and $\ophi=\lambda$. In particular, $f\:Y=\Spec(Q)\to X$ is birational and $g=\of$. In addition, it is clear that $Q$ is finitely generated over $P$ if and only if $R$ is finitely generated over $\oP$.

Furthermore, we claim that given this $\phi$, there is a unique $\psi\:Q\to L$ making the left diagram commutative. Indeed, since $Q\subseteq Q^\gp=P^\gp$, it suffices to check that $\rho^\gp(Q)$ lies in $L$. Therefore, it suffices to check that $P^\gp\to\oL^\gp$ takes $Q$ to $\oL$, and the latter is clear since $Q^\gp=P^\gp\to\oL^\gp$ is compatible with the homomorphism $Q\to R\to\oL$.

It remains to show that $\phi$ is unique. If $\phi'\:P\to Q'$ is another such homomorphism, then applying (ii) to $L=Q'$ and the identity $\oQ=R=\oQ'$, we obtain a canonical map $Q'\to Q$. Its inverse, $Q'\to Q$ is obtained in the same way, hence $\phi$ is unique.
\end{proof}

Since the sharpening functor preserves birationality, Lemma~\ref{birlem}(i) implies the following result.

\begin{cor}\label{bircor}
For any monoscheme $X$ the sharpening functor induces an equivalence between the categories of birational morphisms $Y\to X$ and birational morphisms $Z\to\oX$. This equivalence preserves the property of morphisms to be of (locally) finite type.
\end{cor}

\subsubsection{Subdivisions}
A morphism of fans or monoschemes $f\:Y\to X$ is called a {\em subdivision} (resp. a {\em partial subdivision}) if it is birational, of finite type, and the following condition is satisfied:

($S$) For any valuative monoid $R$ and the associated monoscheme $S=\Spec(R)$, the map $\lambda_S\:\Mor(S,Y)\to\Mor(S,X)$ is bijective (resp. injective).

\begin{rem}
\begin{compactenum}[(i)]
\item Kato used in \cite[Section 9]{Kato-toric} the notions of proper subdivisions and subdivisions. Our terminology follows \cite[Section~3.5]{Gabber-Ramero}, which seems to be more precise. These notions are analogous to the notions of separated birational and proper birational morphisms of schemes, with ($S$) being an analog of the valuative criteria.

\item In \cite[Section 9]{Kato-toric} Kato considered the case of fine fans, and then it sufficed to consider the valuative monoid $\bfN$ only. This is analogous to the fact that for noetherian schemes it suffices to consider traits in the valuative criteria.
\end{compactenum}
\end{rem}

\begin{lem}\label{subdivlem}
In the definition of (partial) subdivisions one can replace condition (S) by the following condition ($\oS$): For any sharp valuative monoid $R$ and the associated fan $\oS=\oSpec(R)$, the map $\lambda_\oS\:\Mor(\oS,Y)\to\Mor(\oS,X)$ is bijective (resp. injective).
\end{lem}
\begin{proof}
This follows easily from the following claim: if $M\to N$ is a homomorphism of monoids such that $M^\gp=N^\gp$, and $M\to R$ is a homomorphism to a valuative monoid, then $M\to R$ factors through $N$ if and only if $M\to\oR$ factors through $N$. The latter is true for any integral monoid $R$. Indeed, $M^\gp\times_{R^\gp}R=M^\gp\times_{\oR^\gp}\oR$ because $R=R^\gp\times_{\oR^\gp}\oR$, and it remains to use that $M\to R$ factors through $N$ if and only if $N\subseteq M^\gp\times_{R^\gp}R$, and similarly for $M\to\oR$.
\end{proof}

\begin{cor}\label{subdivcor}
Let $X$ be a monoscheme with associated fan $\oX$. Then the sharpening functor induces equivalence between the categories of subdivisions (resp. partial subdivisions) of $X$ and $\oX$.
\end{cor}
\begin{proof}
By Corollary~\ref{bircor} we should only prove the following claim: assume that $f\:Z\to X$ is birational, then $f$ is a subdivision (resp. a partial subdivision) if and only if $\of$ is. The claim, in its turn, follows from Lemma~\ref{subdivlem} and the obvious observation that, given a fan $F$, any morphism $F\to Z$ factors uniquely through $\oZ$, that is, $\Mor(F,\oZ)=\Mor(F,Z)$.
\end{proof}

\subsubsection{Local isomorphisms}
Let $f\:Y\to X$ be a morphism of monoschemes (resp. fans). We say that $f$ is a {\em local isomorphism} if it is of finite type and for any $y\in Y$ with $x=f(y)$ one has that $M_{X,x}=M_{Y,y}$ (resp. $\oM_{X,x}=\oM_{Y,y}$). As in the case of schemes, the following result holds:

\begin{lem}\label{locisomlem}
Assume that $f$ is a local isomorphism between monoschemes or fans. Then $f$ is birational, and $f$ is a subdivision (resp. a partial subdivision) if and only if it is bijective (resp. injective).
\end{lem}
\begin{proof}
For concreteness, we consider only the case of monoschemes. Let $R$ be a valuative monoid. Then $S=\Spec(R)$ is a local monoscheme with closed point $s$, and giving a morphism $f\:S\to X$ is equivalent to choosing a point $x=f(s)$ and giving a local homomorphism $M_{X,x}\to R$. In addition, as in the case of rings, using Zorn's lemma it is easy to prove that for any point $x\in X$ there exists a local homomorphism $M_{X,x}\to R$ where $R$ is a valuative monoid. (In addition, one can achieve that $R^\gp=M_{X,x}^\gp$.) The lemma follows.
\end{proof}

\subsubsection{Blow ups}
Let $X$ be a monoscheme and $I\subseteq\calO_X$ a quasi-coherent sheaf of ideals. Then there is a universal morphism $f\:X'\to X$ such that $f^*(I)$ is invertible (that is, locally generated by a single element). One calls $X'$ the {\em blow up} of $X$ along $I$ and denotes it $X_I$, see \cite[Theorem~II.1.7.2]{Ogus-logbook}. The theory of blow ups of monoschemes is developed in \cite[Section~II.1.7]{Ogus-logbook} and it is simpler than its scheme-theoretic analogue. Note also that $M_X$ and $\oM_X$ have the same ideals and hence applying the sharpening functor one obtains a similar theory for fans. For example, $\of\:\overline{X_I}\to\oX$ is the universal morphism of fans such that the pullback of $I\oM_X$ is invertible.

Let us list a few simple facts from \cite[Section~II.1.7]{Ogus-logbook}. Blow ups form a filtered family, since $X_I\times_XX_J=X_{I+J}$. Blow ups are compatible with base changes $h\:Y\to X$, that is, $Y_{h^*(I)}=X_I\times_XY$. If $X=\Spec(M)$ and $I=(f_1\..f_n)$ then $X_I$ is glued from the charts $X_{I,a}=\Spec(A[I-a])$ for $a\in I$, where $A[I-a]$ is the submonoid of $A^\gp$ generated by elements $b-a$ with $b\in I$.

\begin{lem}\label{properblowup}
If $X$ is a monoscheme (resp. a fan) and $I\subseteq M_X$ (resp. $I\subseteq\oM_X$) is a coherent ideal then the blow up $X_I\to X$ is a subdivision.
\end{lem}
\begin{proof}
The claim is local on the base, so we can assume that $X=\Spec(A)$ and $I=(f_1\..f_n)$. The birationality is clear. If $R$ is a valuative monoid and $\phi\:A\to R$ is a homomorphism, then choosing $i$ so that $\phi(f_i)$ is minimal in $R$ we see that $\phi$ factors through $A[I-f_i]$. It follows that $\Spec(R)\to X$ factors through $X_I$ via the $f_i$-chart, in particular, $X_I\to X$ is a subdivision.
\end{proof}

\subsubsection{Exactification}
A homomorphism of monoids $P\to Q$ is {\em exact} if $P=Q\times_{Q^\gp}P^\gp$. A morphism of monoschemes (resp. fans) $f\:Y\to X$ is {\em exact} if the homomorphisms $M_{X,f(y)}\to M_{Y,y}$ (resp. $\oM_{X,f(y)}\to\oM_{Y,y}$) are exact. In particular, an exact birational morphism is nothing else but a local isomorphism. It is easy to see that exact morphisms are preserved by base changes. The following result is a (very simple) analogue of the flattening theorem of Raynaud-Gruson.

\begin{theor}\label{exactlem}
Assume that $X=\Spec(P)$ is an affine monoscheme (resp. a fan) and $f\:Y\to X$ is a morphism of finite type. Then there exists a finitely generated ideal $I\subseteq P$ such that the base change $f_I\:Y_I\to X_I$ of $f$ with respect to the blowing up $X_I\to X$ is an exact morphism. In particular, if $f$ is birational, then $f_I$ is a local isomorphism.
\end{theor}
\begin{proof}
Since the family of blow ups is filtered, we can fix a finite covering of $Y$ by affines $Y_i$ and prove the claim separately for each $Y_i$. So, we can assume that $Y$ is affine, and then $f$ is a base change of a morphism of fine monoschemes $Y_0\to X_0$ by Lemma~\ref{approxlem}. This reduces us to the case when $X$ is fine, and the latter is proved in \cite[Theorem~II.1.8.1]{Ogus-logbook}.
\end{proof}

Note that the same result holds when $X$ is quasi-compact, but we prefer not to work this out here.

\subsection{Toric monoids over valuative monoids}\label{toricmonsec}

\subsubsection{Notation}\label{Rsec}
In the sequel, $R$ is a sharp valuative monoid of height one. For simplicity we assume that $R$ is divisible, and hence $R^\gp$ is a vector space over $\bfQ$. In addition, we fix an ordered embedding $R^\gp\into V_\bfR:=\bfR$ and denote its image by $V$, thus identifying $R$ and $V_{\ge 0}$. Note that $\emptyset$ and $\mathfrak{m}_R=R\setminus\{0\}$ are the only prime ideals of $R$, so $\Spec(R)$ consists of a generic point $\eta=\Spec(V)$ and a closed point $s$. %Their images in $\oSpec(R)$ will be denoted $\oeta$ and $\os$.
For any $R$-monoscheme $f\:X\to\Spec(R)$ we set $X_s=f^{-1}(X_s)$ and $X_\eta=f^{-1}(\eta)$.

%Our goal is to extend classical results about toric monoids, cones and fans to $R$-monoids. The classical results correspond to the absolute case, when $R=0$ is of height zero, and they will be used to prove the generalizations to the relative case of $R$-monoids.

\subsubsection{Faithful $R$-monoids}\label{faithmonsec}
Let $M$ be an $R$-monoid with the structure homomorphism $\phi\:R\to M$. In this case, one of the following possibilities holds:

\begin{compactenum}[(1)]
\item The homomorphism $\phi$ is local: $\phi^{-1}(M^\times)=0$. In this case we say that $M$ is {\em faithful} over $R$. This happens if and only if the closed fiber of $\Spec(M)\to\Spec(S)$ is non-empty.

\item The monoid $M$ is not faithful, and then $\phi^{-1}(M^\times)=R$. This can happen in two cases:
\begin{compactenum}[(a)]
\item $\phi$ is not injective, in which case we say that $M$ is {\em non-flat} over $R$.

\item $\phi$ factors as $R\into V\into M$, in which case we say that $M$ is flat but not faithful over $R$.
\end{compactenum}
\end{compactenum}

\subsubsection{Faithful monoschemes}
We say that an $R$-monoscheme or an $R$-fan $X$ is {\em faithful} if any its point has a specialization in $X_s$. In other words, $X$ is glued from spectra of $R$-faithful monoids.

\subsubsection{$R$-toric monoids}
By an {\em $R$-toric} monoid we mean a flat saturated torsion-free $R$-monoid $P$ finitely generated over $R$. In particular, if $R=0$ then $P$ is nothing else but a classical toric monoid. In the same venue, by an $R$-toric fan or a monoscheme we mean the one glued from finitely many spectra of $R$-toric monoids. Products in the categories of $R$-toric fans and monoschemes correspond to the saturated $R$-pushout. For example, the fiber product of $\Spec(M)$ and $\Spec(N)$ is $\Spec((M\oplus_RN)^\sat)$.

\subsubsection{$V$-lattices}
By a {\em $V$-lattice} we mean a torsion-free abelian group $L$ which contains $V$ and is finitely generated over it. In particular, $L$ is an $R$-toric monoid. Note also that if $P$ is an $R$-toric monoid then $P^\gp$ is a $V$-lattice. We define the {\em $V$-rank} $\rk_V(L)$ of $L$ to be the $\bfZ$-rank of the lattice $L/V$.

\begin{lem}\label{Rtoriclem}
Let $L$ be a $V$-lattice and let $L_0$ be a subgroup of $L$.

\begin{compactenum}[(i)]

\item If $V\subseteq L_0$, then $L_0$ is a $V$-lattice, and $L/L_0$ is a lattice if and only if $L_0$ is saturated in $L$. In the latter case, $L$ splits non-canonically as $L_0\oplus L/L_0$. In particular, $L\toisom V\oplus \bfZ^n$, where $n=\rk_V(L)$.

\item If $L_0\cap V=0$, then $L_0$ is a lattice, and $L/L_0$ is a $V$-lattice if and only if $L_0$ is saturated in $L$. In the latter case, $L$ splits non-canonically as $L_0\oplus L/L_0$.
\end{compactenum}
\end{lem}
\begin{proof}
For (i), note that $L_0$ is a $V$-lattice and $L/L_0$ is a finitely generated. So, $L/L_0$ is a lattice if and only if it is torsion-free, which happens if and only if $L_0$ is saturated in $L$. The splitting is clear since $L/L_0$ is free.

For (ii), we note that everything except splitting is, again, very simple. To prove splitting, we should find a section of $L\onto L/L_0$. Since $L/L_0=V\oplus \bfZ^n$ by (i), it suffices to find sections of the composed maps $L\onto L/L_0\onto V$ and $L\onto L/L_0\onto\bfZ^n$. The first one is canonical, and the second one exists by the freeness of $\bfZ^n$.
\end{proof}

\begin{cor}\label{Rtoriccor}
Any $R$-toric monoid $P$ non-canonically splits as $P=P^\times\oplus\oP$. In addition,
\begin{compactenum}[(i)]
\item If $P$ is faithful, then $P^\times$ is a lattice disjoint from $V$ and $\oP$ is an $R$-toric monoid.

\item If $P$ is not faithful, then $P^\times$ is a $V$-lattice and $\oP$ is a toric monoid.
\end{compactenum}
\end{cor}
\begin{proof}
Only the splitting requires an argument. In particular, it is clear that $P^\times$ is either a lattice or a $V$-lattice. So, $P^\times$ splits off $P^\gp$ by Lemma~\ref{Rtoriclem}, and it remains to use Lemma~\ref{unitmonlem}.
\end{proof}

In particular, the corollary reduces studying $R$-toric monoids to studying sharp ones. Furthermore, if $P$ is not faithful then $\oP$ is toric, and the theory of such objects is classical. For this reason, the non-faithful case will be excluded in the sequel each time it requires a special consideration.

\subsubsection{Basic and smooth monoids}
Next, we will split toric monoids to further simpler blocks. We say that a monoid $P$ is {\em basic} if it does not contain prime elements.

\begin{lem}\label{basiclem}
Any $R$-toric monoid $P$ splits as $P^b[t_1\..t_n]$, where $P^b$ is basic. Moreover, $(t_1)\..(t_n)$ are precisely the principal prime ideals of $P$ and $P^b=P\setminus\cup_{i=1}^n(t_i)$. In particular, $P^b$ is uniquely determined and the elements $t_i$ are unique up to units.
\end{lem}
\begin{proof}
Note that $\Spec(P)$ is finite because $P$ is finitely generated over $R$, and hence $n$ is bounded. So, a decomposition exists by Lemma~\ref{freeofflem}, and the uniqueness claim is easily checked.
\end{proof}

Combining Lemma~\ref{basiclem} with Corollary~\ref{Rtoriccor} we obtain a splitting $P=P^\times\oplus \oP^b[t_1\..t_n]=P^\times\oplus\bfN^n\oplus\oP^b$. The fact that $Q^\times=(Q[t])^\times$ for any monoid $Q$ implies that $(\oP)^b=\overline{P^b}$, so there is no ambiguity in the notation $\oP^b$. We call $P^b$ and $\oP^b$ the basic component of $P$ and the sharp basic component of $P$, respectively. If $P$ is faithful then $P^\times=\bfZ^l$ and the above splitting can be written in the polynomial form as $P=\oP^b[t_1\..t_n,\pm s_1\..\pm s_l]$.

An $R$-toric monoid $P$ will be called {\em smooth} if $\oP^b=0$. In particular, the splitting of $P$ can be recombined as $P^{sm}\oplus\oP^b$, where $P^{sm}$ is the largest smooth factor of $P$.

\subsubsection{Polystable monoids}\label{polymonsec}
We have reduced study of $R$-toric monoids to study of their sharp basic component. A simplest example of a sharp basic monoid is $P_{n,\pi}=R[t_0\..t_n]/(\sum_{i=0}^nt_i=\pi)$, where $\pi\in \mathfrak{m}_R$ and $n\ge 0$. The monoids with $n>0$ are non-smooth (and non-isomorphic), while $P_{0,\pi}=R$. An $R$-toric monoid $P$ is called {\em semistable} if $\oP^b=P_{n,\pi}$. An $R$-pushout of semistable monoids is called {\em polystable}. A faithful $R$-toric monoid $P$ is polystable if and only if $\oP$ is the $R$-pushout of a free monoid and few $P_{n_i,\pi_i}$. Explicitly,
$$\oP=R[t_1\..t_n,t_{1,0}\..t_{1,n_1}\..t_{l,0}\..t_{l,n_l}]/\left(\sum_{j=0}^{n_1}t_{1,j}=\pi_1\..\sum_{j=0}^{n_l}t_{l,j}=\pi_l\right).$$

The following simple facts will not be used, so we omit a justification.

\begin{rem}
\begin{compactenum}[(i)]
\item The isomorphism class of a faithful polystable $P$ is determined by the rank of $P^\times$, the rank $n$ of the free part, and the unordered tuple of pairs $((n_1,\pi_1)\..(n_l,\pi_l))$ with $n_i>0$.

\item Since $P_{n,\pi}\oplus_RV=V\oplus\bfZ^n$, a non-faithful $P$ is polystable if and only if it is smooth.
\end{compactenum}
\end{rem}

\subsubsection{Simplical monoids}
An embedding of saturated monoids $P\into Q$ is called a {\em Kummer extension} if $Q^\gp/P^\gp$ is finite and $Q$ is the saturation of $P$ in $Q^\gp$. A sharp $R$-toric monoid $Q$ is called {\em simplicial} if it contains a semistable submonoid $P$ such that $Q/P$ is a Kummer extension.

\subsubsection{Types of monoschemes}
Let $\bfP$ be one of the following properties: {\em smooth}, {\em semistable}, {\em polystable}, {\em simplicial}. We say that an $R$-toric monoscheme (resp. a fan) $X$ is $\bfP$ at a point $x\in X$ if the monoid $M_{X,x}$ (resp. $\oM_{X,x}$) is $\bfP$, and $X$ is $\bfP$ if it so at all its points. All these properties are not sensitive to units, hence $X$ is $\bfP$ at $x$ if and only $\oX$ is $\bfP$ at $x$.

\section{Fans and polyhedral complexes}\label{fansec}
Throughout Section~\ref{fansec}, $R$, $V$ and $V_\bfR$ are as in Section~\ref{Rsec}. We develop a basic theory of cones and fans associated with $R$-toric monoids, which is analogous to the usual theory of toric monoids, cones and fans. The latter will be referred to as the ``classical case".

\subsection{$R$-toric monoids and special $V$-cones}%\Michael{Check if some references to Gubler-Soto can be in place.}

\subsubsection{$V$-lattices in real vector spaces}
By a {\em $V$-lattice} in a real vector space $W$ we mean a subgroup $L\subset W$ such that $L$ is a $V$-lattice, $V\subset L$ spans a line $V_\bfR$ in $W$ satisfying $V=L\cap V_\bfR$, and $L/V$ is a lattice in $W/V_\bfR$. This happens if and only if $W=\bfR^{1+n}$ and $L=V\oplus\bfZ^n$ for an appropriate basis of $W$. Note that once a $V$-lattice in $W$ is fixed, we obtain an embedding $R\into W$. For any $\pi\in \mathfrak{m}_R$ its image in $W$ will be denoted $v_\pi$.

Any abstract $V$-lattice $L$ determines a canonical real vector space $L_\bfR$ in which it embeds. For example, using a splitting $L=V\oplus\bfZ^n$ we can take $L_\bfR=V_\bfR\oplus\bfR^n$, and this is easily seen to be independent of the splitting. Alternatively, one can take $L_\bfR=(L\otimes\bfR)\oplus_{V\otimes\bfR}V_\bfR$, where $V\otimes\bfR\to V_\bfR$ is the map induced by the embedding $V\into V_\bfR=\bfR$.

\subsubsection{Special $V$-cones}
By a {\em polyhedral cone with a $V$-integral structure} we mean a real vector space $W_\sigma$ provided with a $V$-lattice $L_\sigma$ and a polyhedral cone $\sigma\subseteq L_\bfR$ with the vertex at 0, a non-empty interior, and edges spanned by elements of $L_\sigma$. If, in addition, $v_\pi\in\sigma$ then we say that $\sigma$ is a {\em special $V$-cone}. This property is independent of the choice of $\pi\in \mathfrak{m}_R$. A morphism of special $V$-cones $\sigma\to\tau$ is a linear map $f\:W_\sigma\to W_\tau$ such that $f(v_\pi)=v_\pi$, $f(L_\sigma)\subseteq L_\tau$ and $f(\sigma)\subseteq\tau$.

\subsubsection{Cones and monoids}
As in the classical case, to any $R$-toric monoid $M$ one can assign a special $V$-cone $\sigma=\sigma_M$ as follows: $L_\sigma=M^\gp$, $W_\sigma=(L_\sigma)_\bfR$, and $\sigma=M_\bfR$ is the convex hull of $M$ in $W_\sigma$. It is spanned by the generators of $M$, and hence is a special $V$-cone. Clearly, the construction $M\mapsto\sigma_M$ is functorial.

Conversely, to any special $V$-cone $\sigma$ one can functorially associate the $R$-monoid $M_\sigma=M(\sigma):=\sigma\cap L_\sigma$. As in the classical case, these constructions are inverse:

\begin{theor}\label{conemonoidth}
The functors $\sigma\mapsto M_\sigma$ and $M\mapsto\sigma_M$ are essentially inverse equivalences between the categories of $R$-toric monoids and special $V$-cones. In particular, for any special $V$-cone $\sigma$, the monoid $M_\sigma$ is $R$-toric.
\end{theor}
\begin{proof}
First, let us prove the following claim for a special $V$-cone $\sigma$: (1) $M_\sigma$ is $R$-toric, (2) if $M\subseteq M_\sigma$ is a submonoid, which is saturated in $L_\sigma$ and contains $V\cap M_\sigma$ and non-zero elements on all edges of $\sigma$, then necessarily $M=M_\sigma$.

Choosing an appropriate basis we can assume that $L_\sigma=V\oplus\bfZ^n$ and $W_\sigma=V_\bfR\oplus\bfR^{n}$. Assume first that $\sigma$ contains $V_\bfR$. Then $\sigma$ splits as $V_\bfR\oplus \sigma'$, where $\sigma'$ is a rational polyhedral cone in $\bfR^n$. In addition, $V\subset M$ and hence $M$ splits as $V\oplus M'$, where $M'$ is saturated in $\bfZ^n$. By the classical results, $M_{\sigma'}=\sigma'\cap\bfZ^n$ is a toric monoid and $M'=M_{\sigma'}$. In particular, $M_\sigma=V\oplus M_{\sigma'}$ is $R$-toric and $M=M_\sigma$.

Assume now that $\sigma$ does not contain $V_\bfR$, and hence $\sigma\cap V_\bfR=R_\bfR$. Note that $\sigma$ is the union of cones $\sigma_i$ spanned by $R_\bfR$ and a face $F_i$ of $\sigma$ which does not contain $R_\bfR$. Therefore, $M_\sigma=\cup_i M_{\sigma_i}$, and it suffices to prove that each monoid $M_{\sigma_i}$ is toric and coincides with $M_i:=M\cap\sigma_i$. This reduces the question to the case of a single cone $\sigma_i$, and to simplify notation we replace $\sigma$ by $\sigma_i$, achieving that $R_\bfR$ is an edge of $\sigma$ and all other edges are contained in a face $F$. Let $L'$ be the intersection of $V\oplus\bfZ^n$ with the hyperplane spanned by $F$. Clearly, $L'$ is a lattice of rank $n$ saturated in $V\oplus\bfZ^n$, and hence $V\oplus\bfZ^n=V\oplus L'$ by Lemma~\ref{Rtoriclem}. So, changing coordinates we can assume that $L'=\bfZ^n$ and $F$ lies in $\bfR^n$. Then $\sigma=R_\bfR\oplus\sigma'$, and hence $M_\sigma=R\oplus M_{\sigma'}$ for $M_{\sigma'}=\sigma'\cap L'$. By the classical results, $M_{\sigma'}$ is toric and coincides with $M\cap L'$. In particular, $M_\sigma=M$ is $R$-toric.

The rest is easy. If $\sigma$ is a special $V$-cone then $M_\sigma$ contains non-zero vectors on the edges of $\sigma$, and hence $\sigma_{M(\sigma)}=\sigma$. Conversely, if $M$ is an $R$-toric monoid then the inclusion $M\subseteq M(\sigma_M)$ is an equality by claim (2).
\end{proof}

\subsubsection{Sums of special $V$-cones}\label{sumconesec}
The usual direct sum $L_1\oplus L_2$ of $V$-lattices is not a $V$-lattice because it contains $V^2$. Naturally, this problem should be remedied by switching to $V$-sums $L=L_1\oplus_V L_2$, which is easily seen to be a $V$-lattice. (In fact, $L=(L_1\oplus L_2)\oplus_{V^2}V$, so it is obtained from $L_1\oplus L_2$ by replacing $V^2$ with $V$.) Similarly, instead of direct sums of special $V$-cones $\sigma$ and $\tau$ one should consider $\rho=\sigma\oplus_V\tau$ given as follows: $L_\rho=L_\sigma\oplus_V L_\tau$ and $\tau$ is the image of $\sigma\oplus\tau$ under the projection $W_\sigma\oplus W_\tau\to W_\rho$. One immediately checks that the equivalence $M\mapsto\sigma_M$ respects sums: $\sigma(M\oplus_R N)=\sigma_N\oplus_V\sigma_N$.

\subsubsection{Faithful cones}
Similarly to faithful monoids (\S\ref{faithmonsec}), we say that a special $V$-cone $\sigma$ is {\em faithful} if $\sigma\cap V=R$. Clearly the equivalences of Theorem~\ref{conemonoidth} take faithful monoids to faithful cones. In fact, the case of non-faithful objects in this theorem is not so interesting because, as we saw in the proof of the theorem, the claim reduces to the classical case after factoring by $V$ and $V_\bfR$.

\subsubsection{Sharp cones}
We say that a polyhedral cone $\tau$ is {\em sharp} if it does not contain lines (often one says that $\tau$ is strongly convex). If $\tau^\times$ denotes the maximal vector subspace of $\tau$, then $\otau:=\tau/\tau^\times$ is sharp and $\tau$ non-canonically splits as $\tau^\times\oplus\otau$. Note that a splitting $M=M^\times\oplus\oM$ of an $R$-toric monoid gives rise to such a splitting $M_\bfR=M^\times_\bfR\oplus\oM_\bfR$.

\subsubsection{Simplicial cones}\label{simpconesec}
Let $\sigma$ be a $V$-special cone. We call $\sigma$ {\em simplicial} if the sharpening $\osigma$ is a simplicial polyhedral cone.

\begin{lem}\label{simpconelem}
Let $\sigma$ be a faithful $V$-special cone. Then

\begin{compactenum}[(i)]
\item $n=\dim(W_\osigma)-1\ge 0$,

\item $\sigma$ is simplicial if and only if it is spanned by $\sigma^\times$ and $n+1$ elements $t_0\..t_n$.

\item The elements $t_i$ in (ii) are linearly independent over $\sigma^\times$, and one can choose them so that $t_i$ are not divisible by integers $m>1$ and $\sum_{i=0}^na_it_i=v_\pi$ for $\pi\in \mathfrak{m}_R$ and $a_i\in\bfN$ such that $(a_0\..a_n)=1$. Furthermore, $a_i$ are unique up to permutation and $t_i$ are unique up to permutation and shifting by elements of $L_{\sigma^\times}=L_\sigma\cap\sigma^\times$.
\end{compactenum}

\end{lem}
\begin{proof}
By the faithfulness of $\sigma$, the image of $V_\bfR$ in $W_\osigma=W_\sigma/\sigma^\times$ is non-zero, and hence (i) holds. Choose  $\ot_0\..\ot_n\in L_\osigma$ that generate the edges of $\osigma$, then their lifts $t_i\in L_\sigma$ satisfy (ii) and are linearly independent over $\sigma^\times$. The image of $R_\bfR$ lies in $\osigma$, hence there exists $\pi\in \mathfrak{m}_R$ such that $\sum_{i=0}^na_i\ot_i=\ov_\pi$ for some choice of $a_i\in\bfN$. Dividing the equation by the gcd of the coefficients, we can also achieve that $(a_0\..a_n)=1$. Shifting $t_0$ by an element of $L_{\sigma^\times}$, we achieve that also $\sum_{i=0}^na_it_i=v_\pi$. The uniqueness is clear from the construction.
\end{proof}

\subsubsection{Semistable cones}\label{semiconesec}
A $V$-special cone $\sigma$ is {\em semistable} if it is faithful and simplicial, and an equality $\sum_{i=0}^na_it_i=v_\pi$ from Lemma~\ref{simpconelem}(iii) satisfies the following two conditions: $a_i\in\{0,1\}$ and $t_i$ generate $L_\sigma$ over $L_{\sigma^\times}$. In particular, renumbering $t_i$ we can achieve that $t_0+\dots+t_m=v_\pi$.

\begin{lem}\label{semiconelem}
A faithful $R$-toric monoid $M$ is simplicial (resp. semistable) if and only if the associated $V$-special cone $\sigma$ is simplicial (resp. semistable).
\end{lem}
\begin{proof}
The case of semistable monoids and cones is obvious. To deal with the simplicial case note that if $M\subseteq N$ is a Kummer extension then $M_\bfR=N_\bfR$, and $\sigma_M$ and $\sigma_N$ are only distinguished by the $V$-lattices $M^\gp$ and $N^\gp$. This reduces the lemma to the following claim: a special $V$-cone $\sigma=(\sigma,L_\sigma,W_\sigma)$ is simplicial if and only if there exists a $V$-sublattice $L'\subseteq L_\sigma$ of finite index such that the special $V$-cone $\sigma'=(\sigma,L',W_\sigma)$ is semistable. To prove the latter it suffices to choose $t_i$ as in Lemma~\ref{simpconelem}(iii) and take $L'$ to be the $V$-lattice generated over $V\oplus L_{\sigma^\times}$ by $t'_i$, where $t'_i=t_i$ if $a_i=0$ and $t'_i=a_it_i$ otherwise.
\end{proof}

\subsection{$R$-monoschemes and special $V^\bot$-cones}
Our next goal is to study geometry of cones in dual spaces.

\subsubsection{Special $V^\bot$-cones}
By a special $V^\bot$-cone $\tau$ we mean a vector space $W_\tau$, a $V$-lattice $L_\tau$ in the dual space $W^*_\tau$, and a sharp cone $\tau$ in $W_\tau$, which is defined by finitely many inequalities $l_i(x)\ge 0$ with $l_i\in L_\tau$. A morphism of special $V^\bot$-cones $\tau\to\rho$ is a linear map $f\:W_\tau\to W_\rho$ such that $f(\tau)\subseteq\rho$ and the dual map $f^*$ takes $L_\rho$ to $L_\tau$.

For any special $V$-cone $\sigma$ we define its dual $\sigma^*$ as follows: $W_{\sigma^*}=W_\sigma^*$, $L_{\sigma}=L_{\sigma^*}$, and $\sigma^*$ is given in $W^*_\sigma$ by the inequalities $l_i(x)\ge 0$ with $l_i\in M_\sigma$. It suffices to take a set $\{l_i\}$ of generators of $M_\sigma$, hence $\sigma^*$ is a special $V^\bot$-cone. In the same way, one defines duals of special $V^\bot$-cones, which are special $V$-cones. The usual duality between cones with non-empty interior and sharp cones in $\bfR^n$ implies the following result:

\begin{lem}\label{duallem}
The duality functors establish essentially inverse equivalences between the category of special $V^\bot$-cones and the category opposite to the category of special $V$-cones.
\end{lem}

\begin{rem}
In the classical toric case, dual cones are also rational, so their geometry is similar. For $V$-cones this is not the case. For example, special $V$-cones possess subdivisions by simplicial special $V$-cones, as follows easily from the argument in the proof of Theorem~\ref{conemonoidth}, but special $V^\bot$-cones may not have such a subdivision. In fact, the main result we will prove about them later states that they possess subdivisions by polystable special $V^\bot$-cones.
\end{rem}

\subsubsection{Products}
Products/sums is another feature that distinguishes $V$-cones and $V^\bot$-cones. As in the case of $V$-cones (see \S\ref{sumconesec}), the direct product of $V^\bot$-cones $\sigma\times\tau$ is not a $V^\bot$-cone. The natural replacement this time is the $V$-product $\rho=\sigma\times_V\tau$ defined as follows: $L_\rho=L_\sigma\oplus_V L_\tau$, in particular, $W_\rho$ is the hyperplane in $W_\sigma\times W_\tau$ given by equalizing the functionals induced from  $v_\pi$ on $W_\sigma$ and  $v_\pi$ on $W_\tau$, and $\rho=(\sigma\times\tau)\cap W_\rho$. Clearly, this operation is dual to $V$-sums: if $\sigma$ and $\tau$ are $V$-special cones then $\sigma^*\times_V\tau^*=(\sigma\oplus_V\tau)^*$.

\subsubsection{Faithfulness}
We say that a special $V^\bot$-cone $\tau$ is {\em faithful} if $\tau^*$ is so. Note that $\tau$ lies in the halfspace given by $v_\pi(x)\ge 0$, and $\tau$ is faithful if and only if it is not contained in the hyperplane $W_0$ given by $v_\pi(x)=0$. The theory of non-faithful special $V^\bot$-cones thus reduces to the theory of rational cones in $W_0$.

\subsubsection{Cones and affine monoschemes}
Combining Theorem~\ref{conemonoidth} and Lemma~\ref{duallem} we obtain that the category of special $V^\bot$-cones is equivalent to the category dual to the category of $R$-toric monoids. Therefore, we obtain a natural equivalence between the categories of affine $R$-toric monoschemes and special $V^\bot$-cones, which associates to $X=\Spec(M)$ the special $V^\bot$-cone $\tau=\tau_X$ with $W_\tau=(M^\gp_\bfR)^*$, $L_\tau=M^\gp$, and $\tau=\sigma_M^*$. Note that localizations $\Spec(M_p)\into X$ correspond to embeddings $\tau_p\into\tau$, where $\tau_p$ is the face in $\tau$ cut off by the conditions $f=0$ for $f\notin p$. In particular, it is easy to see that this establishes a one-to-one correspondence between the faces of $\tau$ and points of $X$, and the points of $X_\eta$ correspond to the faces contained in $W_0$.

\subsubsection{Embedded cone complexes and $R$-monoschemes}
The functor $X\to \tau_X$ takes localizations to face embeddings and hence can be globalized, once we define global cone complexes. Loosely speaking, such a complex is obtained by gluing finitely many special $V^\bot$-cones along face embeddings. However, the procedure is a little subtle, in particular, cones may intersect at few faces.

By a {\em locally embedded $V^\bot$-cone complex} $\tau$ we mean the following datum: a real vector space $W_\tau$, a $V$-lattice $L_\tau$ in $W^*_\tau$, and a topological space $\tau$ with a finite {\em face set} $\{\tau_i\}$ of closed subspaces provided with a structure $(W_\tau,L_\tau,\tau_i)$ of special $V^\bot$-cones so that the following conditions are satisfied: (1) each inclusion $\tau_i\into\tau_j$ of faces agrees with a face embedding of cones in $W_\tau$, (2) all faces of each $\tau_i$ are in the face set, (3) each intersection $\tau_i\cap \tau_j$ is a union of faces of both $\tau_i$ and $\tau_j$. Morphisms of such complexes are defined in the natural way and using Lemma~\ref{duallem} and a gluing argument one immediately obtains the following result.

\begin{theor}\label{embconicth}
The functor $X\mapsto\tau_X$ uniquely extends to an equivalence between the categories of $R$-toric monoschemes and locally embedded $V^\bot$-cone complexes, that will also be denoted $X\mapsto\tau_X$. The maps of cones to affine monoschemes glue to global maps $\tau_X\to X$ such that face interiors of $\tau_X$ are precisely preimages of points of $X$.
\end{theor}

\begin{rem}
\begin{compactenum}[(i)]
\item By definition, the embeddings $\tau_i\into W_\tau$ agree and give rise to a continuous map $\tau\to W_\tau$, whose image is a $V^\bot$-cone complex $\tiltau$ in $W_\tau$. We will call $\tiltau$ {\em embedded $V^\bot$-cone complex}. In a sense, $\tau$ is obtained from $\tiltau$ by multiplying few its cones. Similarly to the theory of schemes, one can naturally define the notion of separated monoschemes, and it is easy to check that $\tau_X$ is embedded if and only if $X$ is separated.

\item Products of $R$-toric monoschemes correspond to $V$-products of complexes. Indeed, this reduces to the affine case, and then products correspond to $R$-pushouts of monoids, which are taken by the functor $M\mapsto \sigma_M$ to $V$-sums of special $V$-cones.
\end{compactenum}
\end{rem}

\subsubsection{Non-embedded cone complexes and $R$-fans}
Next, we would like to prove an analog of Theorem~\ref{embconicth} for $R$-toric fans. In the affine setting this means that we restrict to the subcategory of sharp $R$-toric monoids. Clearly, the corresponding subcategories of cones are the category of sharp special $V$-cones, and the category of special $V^\bot$-cones with non-empty interior, that is, special $V^\bot$-cones such that $\tau$ spans $W_\tau$. Objects of the latter category will be denoted $\otau=(W_\otau,L_\otau,\otau)$. The main difference with the category of arbitrary special $V^\bot$-cones is that for a face embeddings $\rho\into\tau$, the map $W_\orho\into W_\otau$ is injective and the map $L_\otau\to L_\orho$ is surjective, rather than isomorphisms. In addition, a new feature is that a morphism $(W_\otau,L_\otau,\otau)\to(W_\orho,L_\orho,\orho)$ is determined already by the set-theoretical map $\otau\to\orho$.

A {\em (non-embedded) $V^\bot$-cone complex} is a topological space $\otau$ with a finite {\em face set} $\{\otau_i\}$ of closed subspaces provided with a structure $(W_i,L_i,\otau_i)$ of special $V^\bot$-cones with non-empty interior so that the following conditions are satisfied: (1) each inclusion $\otau_i\into \otau_j$ of faces underlies a face embedding of special $V^\bot$-cones, (2) all faces of each $\otau_i$ are in the face set, (3) each intersection $\otau_i\cap \otau_j$ is a union of faces of both $\otau_i$ and $\otau_j$. If each $\otau_i\cap \otau_j$ is a single face (maybe empty) then we say that the complex is {\em strict}.

\begin{rem}
One often works only with strict complexes because this is technically easier. On the algebra-geometric side this corresponds to working with toroidal structures without self-intersections (as in \cite{KKMS}). On the other hand, arbitrary complexes can be easily subdivided to strict ones. For example, in the classical case this can is achieved by the barycentric subdivision.
\end{rem}

\subsubsection{Equivalence with fans}
Given an affine $R$-toric fan $\oX=\oSpec(M)$ consider the special $V^\bot$-cone $\otau_\oX=\otau_{\Spec(\oM)}$. Clearly, $\oM_\bfR$ is sharp and hence $\otau_\oX$ has a non-empty interior. So, $\oX\mapsto\otau_\oX$ is a functor to the category of special $V^\bot$-cones with non-empty interior. It is easy to see that it takes localizations to face embeddings and hence globalizes, yielding the following analog of Theorem~\ref{embconicth}

\begin{theor}\label{conicth}
The functor $\oX\mapsto\otau_\oX$ uniquely extends to an equivalence between the categories of $R$-toric fans and non-embedded $V^\bot$-cone complexes, that will also be denoted $\oX\mapsto\otau_\oX$.
\end{theor}

\subsection{$V$-polyhedral complexes}
In the case of faithful fans or monoschemes, one can replace special $V^\bot$-cones by polyhedral sections. In this section we will work this out in detail.

\subsubsection{The canonical polyhedral section}
A faithful special $V^\bot$-cone $\tau$ is determined by the canonical section $\tau_1=\tau\cap W_1$, where $W_1$ is the affine hyperplane given by $v_\pi(x)=\pi$. Clearly, $\tau_1$ is a sharp polyhedron (i.e. a polyhedron not containing lines), satisfying a natural $V$-integrality condition that we are going to formalize below.

\subsubsection{$V$-affine spaces}
By a $(\bfZ,V)$-affine or simply {\em $V$-affine} real vector space we mean an $n$-dimensional affine real vector space $U$ (i.e. a torsor of $\bfR^n$) provided with a subgroup $A_U$ of the group of affine functions $U\to\bfR$ such that the elements of $V$ are precisely the constant functions in $A_U$, the group $A_U/V=\bfZ^n$ is a lattice of rank $n$ (in particular, $A_U$ is a $V$-lattice), and the functions of $A_U$ distinguish any pair of points of $U$. A $V$-affine map $f\:U\to W$ between $V$-affine real spaces is an affine map $f$ such that $f^*(A_W)\subseteq A_U$. In particular, the elements of $A_U$ are precisely the $V$-affine functions $U\to\bfR$.

\begin{rem}
\begin{compactenum}[(i)]
\item After an appropriate choice of coordinates $t=(t_1\..t_n)$, we have $U=\bfR^n$ and $A_U=V\oplus\bfZ^n$ with $\bfZ$ spanned by the dual basis. The $V$-affine functions are then of the form $f=v+\sum_{i=1}^na_it_i$ with $a_i\in\bfZ$ and $v\in V$. The freedom to choose coordinates is limited to a transformation $t'=At+v$, where $A\in\GL_n(\bfZ)$ and $v\in V^n$. A similar theory of $(S,V)$-affine spaces can be developed for any subring $S\subseteq\bfR$ and $S$-submodule $V\subseteq\bfR$. See also Remark~\ref{polytoprem} below.

\item For any point $u\in U$ the set of functions $f\in A_U$ such that $f(u)=0$ is a lattice of rank bounded by $n$. We say that $u$ is {\em $V$-rational} if the rank is precisely $n$. It is easy to see the set of $V$-rational points is a torsor for the subgroup $V^n$ of $\bfR^n$. For example, use that rational points correspond to splittings $\bfZ^n\to A_U$ of the sequence $V\into A_U\onto\bfZ^n$.
\end{compactenum}
\end{rem}

\subsubsection{Special polyhedra}
By a {\em special $V$-polyhedron} $P$ we mean a $V$-affine real space $U=U_P$ and a sharp polyhedron $P\subseteq U$ defined by finitely many inequalities $f_j\ge 0$, with $f_j\in A_U$. A morphism of special $V$-polyhedrons $P\to Q$ is any $V$-affine map $U_P\to U_Q$ that takes $P$ to $Q$. If $P$ is bounded then we call it a $V$-affine polytope.

\begin{rem}\label{polytoprem}
(i) Let $S\subseteq\bfR$ be a subring, $L\subseteq\bfR$ an $S$-module, $F={\rm Frac}(S)$ the fraction field and $V=L\otimes_SF$. Berkovich introduced in \cite[Section~1]{bercontr2} $(S,L)$-polytopes in $\bfR^n$ (the notation in \cite{bercontr2} are multiplicative). On the set-theoretical level, it is a polytope in $\bfR^n$ whose vertexes lie in $V^n$ and edge slopes lie in $F$. In particular, $(S,L)$-polytopes and $(F,V)$-polytopes are not distinguished set-theoretically, but they do have different groups of affine functions.

(ii) If $V$ is a vector space over $\bfQ$ then a $(\bfZ,V)$-polytope is a $V$-polytope in our sense, where one takes $A_U=V\oplus\bfZ^n$. For comparison, groups of affine functions on $(\bfQ,V)$-polytopes are of the form $V\oplus\bfQ^n$.
\end{rem}

\subsubsection{Products}
Products in the category of special polyhedra correspond to usual products. Namely, $R=P\times Q$ is defined as the usual set-theoretical product (or Minkowski sum of polyhedra) and, in addition, $U_R=U_P\times U_Q$ and $A_R=A_P\oplus_V A_Q$. This agrees with the $V$-product of $V^\bot$-cones: $(\sigma\times_V\tau)_1=\sigma_1\times\tau_1$.

\subsubsection{Embedded polyhedral complexes}\label{embpolsec}
By a {\em locally embedded $V$-polyhedral complex} we mean a $V$-affine real space $U=U_P$ and a topological space $P$ covered by closed subspaces $P_i$ called {\em faces}, which are provided with the structure of special $V$-polyhedra in $U$ and satisfy the following conditions: (1) each inclusion $P_i\into P_j$ of faces underlies a face embedding of special $V$-polyhedra, (2) all faces of each $P_i$ are in the face set, (3) each intersection $P_i\cap P_j$ is a union of faces of both $P_i$ and $P_j$. As in the cone case, $P$ covers a $V$-polyhedral complex in $U_P$ and is obtained from the latter by multiplying some faces.

A morphism $P\to Q$ consists of a continuous map $P\to Q$ and a $V$-affine map $U_P\to U_Q$ compatible with the projections $P\to U_P$ and $Q\to U_Q$. Clearly the canonical section functor $\tau\mapsto\tau_1$ on special $V^\bot$-cones is compatible with face maps hence globalize to an equivalence between the categories of faithful embedded $V^\bot$-cone complexes and embedded $V$-polyhedral complexes. Combining this with Theorem~\ref{embconicth} we obtain

\begin{theor}\label{embeqth}
The functors $X\mapsto\tau_X$ and $\tau\mapsto\tau_1$ induce equivalences of the following three categories: (1) faithful $R$-toric monoschemes, (2) faithful embedded $V^\bot$-cone complexes, (3) embedded $V$-polyhedral complexes.
\end{theor}

\begin{rem}
\begin{compactenum}[(i)] \item One can also easily describe inverse equivalences. In particular, $\tau$ is reconstructed from $P=\tau_1$ as the cone over $P$ with $L_\tau=A_P$, and if $P$ is a polyhedron then $X=\Spec(M)$ is reconstructed by setting $M$ to be the monoid of all non-negative $V$-affine functions on $P$.

\item The correspondences respect products: products of monoschemes over $R$ correspond to $V$-products of $V^\bot$-cone complexes and products of $V$-polyhedral complexes. All languages are equivalent, but it is often more natural to work within the framework of polyhedral complexes because then products become the most natural ones.
\end{compactenum}
\end{rem}

\subsubsection{Non-embedded polyhedral complexes}
Similarly to cone complexes, one can also define a non-embedded version as follows. First, one restricts to the subcategory of special $V$-polyhedra with non-empty interior and modifies the notion of face embedding accordingly. Then one copies the definition of cone complexes: a (non-embedded) $V$-polyhedral complex is a topological space $P$ with a finite {\em face set} $\{\oP_i\}$ of subspaces provided with a structure $(W_i,L_i,\oP_i)$ of special $V$-polyhedra so that conditions (1)--(3) as in \S\ref{embpolsec} are satisfied.

As in the embedded case, the functor $\otau\mapsto\otau_1$ preserves face embeddings, and using Theorem~\ref{conicth} we obtain

\begin{theor}\label{nonembeqth}
The functors $\oX\mapsto\otau_\oX$ and $\otau\mapsto\otau_1$ induce equivalences of the following three categories: (1) faithful $R$-toric fans, (2) faithful non-embedded $V^\bot$-cone complexes, (3) non-embedded $V$-polyhedral complexes.
\end{theor}

\subsection{Equivalence of properties}\label{equivsec}
In this section we study how different properties are transformed by the equivalences we have constructed.

\subsubsection{Subdivisions}
Subdivisions are respected by these equivalences.

\begin{theor}\label{equivsubdiv}
Assume that $f$ is a morphism of faithful $R$-toric monoschemes (resp. fans) and let $\phi$ and $\phi_1$ be the corresponding morphisms between (resp. non-embedded) cone and polyhedral complexes. Then the morphisms $f$, $\phi$ and $\phi_1$ are subdivisions or partial subdivisions if and only if at least one of them is a subdivision or a partial subdivision, respectively.
\end{theor}

\begin{proof}
The two cases are similar, so we will work with fans for concreteness. The equivalence is obvious for $\phi$ and $\phi_1$, hence it suffices to compare a morphism of fans $f\:Y\to X$ and the induced morphism $\phi\:\tau\to\rho$ of cone complexes. Furthermore, the property of being a (partial) subdivision is local on the target, hence we can assume that $X=\oSpec(M)$. Also, we claim that $f$ is birational if and only if $\phi$ is an embedding locally on $\tau$. Indeed, this is a local claim, hence we can assume that $Y=\oSpec(N)$ is also affine. Then $Y\to X$ is birational if and only if $M^\gp\onto N^\gp$. By duality, this happens if and only if $W_\tau\into W_\rho$, that is $\phi$ is an embedding of polyhedra. So, we can assume in the sequel that $f$ is birational.

Assume first that $Y=X_I$ is the blow up of $X$ along an ideal $I=(f_1\..f_n)$. Then a direct inspection shows that $\phi$ is a subdivision. Indeed, $\tau$ is covered by the polytopes $\tau_i$ corresponding to the charts $\oSpec(M[I-f_i])$. So, $\tau_i$ is cut off from $\rho$ by the inequalities $f_1\ge f_i\..f_n\ge f_i$, and hence $\tau$ is a subdivision of $\rho$.

Returning to the general case, assume that $X'=X_I$ is a blow up, $Y'=Y_I$ its pullback to $Y$, and $\phi'\:\tau'\to\rho'$ the induced morphism of cone complexes. Since $X'\to X$ and $Y'\to Y$ are subdivisions by Lemma~\ref{properblowup}, $f$ is a (partial) subdivision if and only if $f'\:Y'\to X'$ is so. Since $\tau'\to\tau$ and $\rho'\to\rho$ are subdivisions, $\phi$ is a (partial) subdivision if and only if $\phi'$ is so. Therefore, it suffices to prove the theorem for $f'$ instead of $f$, and by use of Theorem~\ref{exactlem} we can assume that $f$ (and hence also $\phi$) is a local isomorphism. By Lemma~\ref{locisomlem}, in the latter case the assertion reduces to the obvious observation that $f$ is injective or bijective if and only if $\phi$ is so.
\end{proof}

\begin{rem}
In fact, the faithfulness assumption is only needed to compare $f$ and $\phi$ to $\phi_1$. Our comparison of $f$ and $\phi$ applies in the non-faithful case too.
\end{rem}

\subsubsection{Polystable complexes}
Let $\tau$ be a faithful special $V^\bot$-cone and $P=\tau_1$ the corresponding special $V$-polyhedron $P$. We say that $P$ and $\tau$ are {\em semistable} (resp. {\em simplicial}) if the dual $V$-special cone $\tau^*$ is so. We say that $P$ (resp. $\tau$) is {\em polystable} if it is a product (resp. a $V$-product) of semistable ones.

\begin{rem}\label{simpolrem}
\begin{compactenum}[(i)] \item Clearly, $\tau$ is simplicial if and only if the underlying cone is simplicial. Note that $\tau$ may have edges (or even a facet) parallel to the $V$-affine space of $P$, and in this case $P$ is an ``unbounded simplex''.

\item Unravelling the definitions one can describe the semistability condition for $\tau$ and $P$ explicitly. For example, $P$ is semistable if and only if there exist $V$-affine functions $t_0\..t_n,s_1\..s_l\in U_P$ such that $t_0+\dots+t_m=\pi$ for some $m\le n$ and $\pi\in \mathfrak{m}_R$, the polyhedron $P$ is given by the conditions $t_i\ge 0$ and $s_j=0$ for $0\le i\le n$ and $0\le j\le l$, and $A_U=V\oplus\bfZ^{n+l}$ with the second factor spanned by $t_1\..t_n,s_1\..s_l$.
\end{compactenum}
\end{rem}

A $V^\bot$-cone complex (resp. a $V$-polyhedral complex) is called {\em simplicial}, {\em semistable} or {\em polystable} if all its cones (resp. polyhedra) are so. This applies both to the locally embedded and non-embedded versions.

\begin{lem}\label{simpollem}
Let $X$ be an $R$-toric monoscheme or fan, let $\tau$ be the associated special $V^\bot$-complex, and let $P$ be the associated $V$-affine polyhedral complex. Then $X$ is simplicial, semistable or polystable if and only if $\tau$ is so.
\end{lem}
\begin{proof}
For concreteness we consider the case of monoschemes. The claim is local, hence we can assume that $X=\Spec(M)$. By Lemma~\ref{semiconelem}, $M$ satisfies any of these properties if and only if $\sigma_M$ does, and by definition the latter happens if and only if $\tau=(\sigma_M)^*$ and $P=\tau_1$ satisfy the same property.
\end{proof}

\subsubsection{Rational polyhedrons}
Our construction of polystable subdivisions will be done by first achieving a weaker property that we are going to define now. Note that any special $V$-polyhedron $P$ can also be viewed as a $(\bfQ,V)$-polyhedron or a rational $V$-polyhedron, that we denote $P_\bfQ$. The functor $P\mapsto P_\bfQ$ is faithful and essentially surjective but far from begin full. This happens because there are many rational but non-integral affine maps, for example, $A_{P_\bfQ}=A_P\otimes\bfQ$, so one usually has that $\Mor(P,\bfR)\subsetneq\Mor(P_\bfQ,\bfR)$.

\subsubsection{Rational splittings}\label{ratsplitsec}
The functor $P\mapsto P_\bfQ$ preserves products, but it may freely happen that a non-invertible morphism $Q\times R\to P$ induces an isomorphism $Q_\bfQ\times R_\bfQ\toisom P_\bfQ$. Thus, if $P$ splits rationally into a product, there might be obstacles to lift it to a splitting of $P$ itself. In fact, if $P_\bfQ=\prod_{i=1}^nQ_i$ for rational $V$-polyhedra, then $A_P\otimes\bfQ=A_{Q_1}\oplus_V\dots\oplus_VA_{Q_n}$ and projections of the $V$-lattice $A_P$ onto the summands give rise to $V$-lattices $A_i$ that refine $Q_i$ to $V$-polyhedra $P_i=(Q_i,A_i)$. The induced map $A_1\oplus_V\dots\oplus_VA_n\to A_P$ is an injective map between $V$-lattices of the same rank, hence it is of finite index $r$ which we call the {\em index} of the rational splitting. Note that $r$ measures how far the rational splitting is from being liftable to an integral one, in particular, $r=1$ if and only if the map $P\to\prod _{i=1}^nP_i$ is an isomorphism. To compute the index one can also switch to the usual $\bfZ$-lattices of affine functions modulo the constant ones: $r$ is the index of $\oplus_{i=1}^nA_i/V$ in $A_P/V$.

\begin{exam}\label{ratsplitexam}
Let $q,r\in\bfR$ be linearly independent over $\bfQ$ and $V=q\bfZ\oplus r\bfZ$. In the standard $V$-affine plane consider the $V$-polytope $P$ with vertices $(0,0)$, $(q,q)$, $(r,-r)$, $(q+r,q-r)$, and let $Q$ and $R$ be the intervals between the origin and the points $(q,q)$ and $(r,-r)$, respectively. Clearly, $P=Q\times R$ rationally, but the index is two because $(1,1)$ and $(1,-1)$ generate a sublattice of index 2 in $\bfZ^2$.
\end{exam}

\begin{rem}
An analog of the above phenomenon exists in the theory of abelian varieties. It usually happens that an abelian variety $A$ can be split into a product of abelian varieties only after replacing it with an isogeneous one. The rank in our case is an analog of the degree of a minimal isogeny needed for splitting.
\end{rem}

\subsubsection{Polysimplicial polyhedrons}
We say that a special $V$-polyhedron $P$ is {\em polysimplicial} if $P=\prod_{i=1}^nP_i$, where $P_i$ are simplicial $V$-polyhedrons. If such a decomposition exists only on the level of rational $V$-polyhedra and $r$ is its index, then we say that $P$ is {\em rationally polysimplicial of index} $r$.

\section{Polystable subdivisions of $(\bfZ,V)$-polytopes and $(\bfZ,V)$-polyhedral complexes}\label{combsec}

We now come to the proof of the main result. To help the reader, we repeat the definitions of the last sections in a combinatorial setting.

\subsection{Polytopes and Minkowski sums}
A \emph{polytope} is the convex hull of a finite set of points in $\bfR^n$. Alternatively, it is a bounded set of the form $\{ x \in \bfR^n : Ax \le b \}$ where $A \in \bfR^{m \times n}$ and $b \in \bfR^m$.

We use the standard lattice in $\bfR^n$ and to define the group of $V$-affine functions, where $V\subset \bfR$ is any $\bfQ$ vectorspace. With this convention, $(\bfZ,V)$-polytopes are simply polytopes with vertices in $V$ and rational facet slopes.

The \emph{dimension} of a polytope is the dimension of the smallest affine subspace containing the polytope. A \emph{simplex} is the convex hull of a set of affinely independent points. The \emph{Minkowski sum} of sets $A_1$, \dots, $A_r \subset \bfR^n$ is the set
\[
A_1 + \dotsb + A_r := \{ x_1 + \dotsb + x_r : x_i \in A_i \text{ for all } i \}.
\]
A \emph{product of simplices}, also \emph{polysimplex}, is a polytope of the form $P = A_1 + \dotsb + A_r$ where each $A_i$ is a simplex and  $\dim(A_1) + \dotsb + \dim(A_r) = \dim(P)$.

A \emph{face} of a polytope $P \subset \bfR^n$ is a set $F \subset P$ such that there exists a linear functional $\phi : \bfR^n \to \bfR$ and $c \in \bfR$ such that $\phi(x) = c$ when $x \in F$ and $\phi(x) < c$ when $x \in P \setminus F$. When this is true, we say that such a pair $(\phi,c)$ \emph{supports} $F$. By this definition, the empty set and $P$ are both faces of $P$. A \emph{vertex} is a face of dimension 0, an \emph{edge} is a face of dimension 1, and a \emph{facet} is a face of dimension $\dim(P)-1$. The \emph{graph} of $P$ is the 1-skeleton of $P$. We use $S(P)$, $E(P)$, and $G(P)$ to denote the vertex set, edge set, and graph of $P$, respectively.

A \emph{polyhedral complex} is a finite collection of $k$-dimensional polytopes in $\bfR^n$, called \emph{cells}, such that the intersection of any two cells is a face of both. Given a polyhedral complex $X$, we let $S(X)$ and $E(X)$ denote the union of the vertex sets and edge sets, respectively, over all cells of $X$. A \emph{subdivision} of a polytope $P$ is a polyhedral complex whose union of cells equals $P$.

\newcommand{\KK}{\bfQ}
\subsubsection{Polystable subdivisions of $(\bfZ,V)$-polytopes}
Let us first note that the class of $(\bfZ,V)$-polytopes is closed under Minkowski sum.

\begin{prop} \label{Vsum}
If $P_1$, \dots, $P_r$ are $(\bfZ,V)$-polytopes, then $P_1 + \dotsb + P_r$ is a $(\bfZ,V)$-polytope.
\end{prop}

Recall that a \emph{polystable subdivision} (or polystable refinement) is a $(\bfZ,V)$-subdivision all of whose cells are polystable as $(\bfZ,V)$-polytopes, that is, each cell is a $(\bfZ,V)$-polysimplex and finally the facet normals of the facets of a cell around any vertex of the same form a unimodular matrix. We call such a polysimplex also simply polystable.

Our goal is to prove the following.

\begin{theor} \label{polysubdv}
Every $(\bfZ,V)$-polytope has a polystable subdivision.
\end{theor}

\subsubsection{A counterexample for semistability}\label{ssec:semis}
Before we continue with the proof of Theorem~\ref{polysubdv}, let us quickly discuss why we cannot hope to obtain a triangulation (into simplices) of rational facet slopes in general. We have the following lemma

\begin{lem}
Consider a $d$-simplex in $\bfR^d$ with rational facet slopes. If $d$ vertices of the simplex are rational, then all vertices of the simplex are rational.
\end{lem}

We conclude:

\begin{prop}
A $d$-polytope $P$ has a triangulation by simplices with rational facet slopes if and only if it is a dilation of a rational polytope.
\end{prop}

\begin{proof}
If $P$ is a dilation of a rational polytope, providing a rational triangulation is an easy exercise. If, on the other hand, $Q$ is a subdivision of $P$ into simplices of rational facet slopes, dilate and translate it and $P$ so that one of the edges has only rational vertices. Then all the remaining vertices of incident simplices are rational as well. Since any two simplices of $Q$ are connected by a path of simplices such that subsequent elements intersect in a simplex of dimension $d-1$. But then all vertices in this dilation are rational. Hence $P$ is the dilation of a rational polytope itself.
\end{proof}

It follows that, for instance, every rectangle with irrational ratio of side lengths cannot be triangulated with triangles of rational slope.

\subsection{Mixed subdivisions}

\subsubsection{Mixed subdivisions}
We now follow \cite[Section~9]{LRS}. Let $P_1$, \dots, $P_r \subset \bfR^n$ be polytopes and $P = P_1 + \dotsb + P_r$. A \emph{mixed subdivision} of $P$ with respect to $P_1$, \dots, $P_r$ is a subdivision of $P$ where each cell $C$ is given a label $(C_1,\dotsc,C_r)$ such that the following hold:
\begin{compactenum}[(i)]
\item  For each $i$, $C_i$ is a polytope with $S(C_i) \subseteq S(P_i)$.
\item $C = C_1 + \dotsb + C_r$.
\item If $C$, $C'$ are two cells labeled $(C_1,\dotsc,C_r)$ and $(C_1',\dotsc,C_r')$, then for each $i$, $C_i \cap C_i'$ is a face of both $C_i$ and $C_i'$.
\end{compactenum}
A mixed subdivision is \emph{fine} if for every cell $C$ with label $(C_1,\dotsc,C_r)$, each $C_i$ is a simplex and $\dim(C_1) + \dotsb + \dim(C_r) = \dim C$.

\begin{exam}
Assume that $P$ is a simplex in $\bfR^n$ with vertices $A_0\.. A_n$. Then the Minkowski sums $Q_i$, $0\le i\le n$ of the simplex with vertices $A_0\.. A_i$ and the simplex with vertices $A_{i}\.. A_{n}$ form a fine mixed subdivision of $2P=P+P$.
\end{exam}

Mixed subdivisions arise naturally from subdivisions of Cayley polytopes, that we introduce in the next section.

\subsubsection{Cayley polytopes} \label{ssec:cayley}
The geometric construction of the previous section is easily generalized to
higher dimensions.  For a family $P_{[r]} = (P_1,\dots,P_r)$ of $r$
polytopes in $\bfR^d$, we define the \emph{Cayley polytope} as
\[
    \Cay(P_{[r]}) \ := \ \conv\Biggl(\,\bigcup_{i=1}^r P_i \times
    e_{i}\,\Biggr) \ \subseteq \ \bfR^d \times \bfR^r.
\]
The coordinate projection $\bfR^d \times \bfR^r \rightarrow \bfR^r$ restricts to a
linear projection
\begin{equation}\label{eq:projection}
    \pi : \Cay(P_{[r]}) \ \longrightarrow \ \Delta_{r} \ = \ \conv\{e_1,\dots,e_r\}
\end{equation}
of the Cayley polytope to the (geometric) standard $(r-1)$-simplex $\Delta_{r}$. It is easy
to see that for $\lambda = (\lambda_1,\dots,\lambda_r) \in \Delta_{r}$, we
have
\begin{equation}\label{eq:fib}
 \pi^{-1}(\lambda) \ \cong \ \lambda_1P_1 + \cdots + \lambda_r P_r.
\end{equation}
Triangulations of Cayley polytopes using only the vertices of summands restrict therefore to fine mixed subdivisions of slices. Moreover, any triangulation of a Cayley polytope that only uses vertices lying in the individual summands induces a subdivision into polysimplices. 

The proposition suggests that the boundary of the Cayley polytope
$\Cay(P_{[r]})$ is stratified along the facial structure of the
cardinality $r$ simplex $\Delta_{r}$. We define the \emph{Cayley complex} $\mathrm{T}_{[r]} =
\mathrm{T}(P_{[r]})$ as the closure of $\pi^{-1}(\mathrm{relint} \Delta_{r}) \cap \partial
\Cay(P_{[r]})$.

\subsubsection{Regular subdivisions} \label{sec:regsubdv}

Let $P \subset \bfR^n$ be a polytope and let $S \subset \bfR^n$ be a finite set such that $\conv(S) = P$, where ``$\conv$'' denotes convex hull. Let $f : S \to \bf R$ be any function. The \emph{lift} of $P$ with respect to $f$ is the polytope
\[
P^f := \conv \{ (s,f(s)) \in \bfR^{n+1} : s \in S \}.
\]
A face $F$ of a polytope is a \emph{lower face} if it is supported by $(\phi,c)$ where $\phi(0,\dotsc,0,1) < 0$. Let $\pi : \bfR^{n+1} \to \bfR$ be the projection map onto the first $n$ coordinates. Then for any lift $P^f$, the set
\[
\{ \pi(F) : F \text{ is a lower facet of } P^f \}
\]
is the set of cells of a subdivision of $P$. We say that this subdivision is \emph{induced} by $f$. A subdivision is \emph{regular} if it is induced by some $f$.

Let $P_1$, \dots, $P_r \subset \bfR^n$ be polytopes and let $P = P_1 + \dotsb + P_r$. For each $i$, let $f_i : S(P_i) \to \bfR$ be any functions. Then the set
\[
\{ \pi(F) : F \text{ is a lower facet of } P_1^{f_1} + \dotsb + P_r^{f_r} \}
\]
is a regular mixed subdivision of $P$. We say that this mixed subdivision is \emph{induced} by $f_1$, \dots, $f_r$. If the $f_i$ are generic, then the induced mixed subdivision is fine. In particular, this implies that there exists a fine mixed subdivision of $P$ with respect to $P_1$, \dots, $P_r$.

\subsection{Proof of Theorem~\ref{polysubdv}, preparation:}

\subsubsection{Construction of polytopes from edges}

We first prove the following weaker Lemma

\begin{lem}\label{polysimp}
Every $(\bfZ,V)$-polytope has a polysimplicial subdivision.
\end{lem}

Here, a polysimplicial subdivision is a subdivision into $(\bfZ,V)$-polysimplices, that is, polysimplices with rational facet normals and vertex coordinates described by $V$.

%\subsubsection{Normal fans} Let $P \subset \bfR^n$ be a polytope and $F$ a face of $P$. The \emph{normal cone} of $P$ at $F$ is the closure of the set of all $\phi \in (\bfR^n)^\ast$ such that $(\phi,c)$ supports $F$ for some $c \in \bf R$. The collection of all normal cones of $P$ over all faces of $P$ is the \emph{normal fan} of $P$.

%\begin{prop}
%If $P_1$, \dots, $P_r$ are polytopes with the same normal fan $\mathcal N$, then $P_1+\dotsb+P_r$ has normal fan $\mathcal N$.
%\end{prop}

Let $P \subset \bfR^n$ be a polytope. Let $F$ be a 2-dimensional face of $P$ with vertices $v_1$, \dots, $v_k$ and edges $v_1v_2$, $v_2v_3$, \dots, $v_kv_1$ (here $uv$ denotes the segment with endpoints $u$, $v$). Let $\eta : E(P) \to \bfR_{>0}$ be a function. We say that $\eta$ is \emph{2-balanced on $F$} if
\[
\eta(v_1v_2) \frac{v_2-v_1}{\lVert v_2-v_1 \rVert} + \eta(v_2v_3) \frac{v_3-v_2}{\lVert v_3-v_2 \rVert} + \dotsb + \eta(v_kv_1) \frac{v_1-v_k}{\lVert v_1-v_k \rVert} = \vec{0}.
\]
We say that $\eta$ is \emph{2-balanced} (with respect to $P$) if it is 2-balanced on all 2-dimensional faces of $P$.

\begin{prop} \label{2bal}
Let $\eta : E(P) \to \bfR_{>0}$ be 2-balanced with respect to $P$. Then there exists a unique up to translation polytope $P' \subset \bfR^n$ and a graph isomorphism $\psi : G(P) \to G(P')$ such that for each edge $uv$ of $P$, we have
\[
\psi(v) - \psi(u) = \eta(uv) \frac{v-u}{\lVert v-u \rVert}.
\]
\end{prop}

\begin{proof}
Fix a vertex $v_0 \in S(P)$, and define $\psi(v_0) := \vec{0}$. For each $v \in S(P)$, we define $\psi(v)$ as follows: Let $v_0$, $v_1$, \dots, $v_k$, $v$ be a path from $v_0$ to $v$ in $G(P)$. Define
\[
\psi(v) := \eta(v_0v_1) \frac{v_1-v_0}{\lVert v_1-v_0 \rVert} + \eta(v_1v_2) \frac{v_2-v_1}{\lVert v_2-v_1 \rVert} + \dotsb + \eta(v_kv) \frac{v_k-v}{\lVert v_k-v \rVert}.
\]
The 2-balanced condition and the fact that 2-skeleton of a polytope has trivial first homology group implies that $\psi(v)$ is well-defined. It is easy to check that $\psi(S(P))$ is the vertex set of a polytope $P'$. Moreover, $\psi$ is a graph isomorphism from $G(P)$ to $G(P')$. Finally, $P'$ is clearly the unique polytope with $\psi(v_0) = \vec{0}$ satisfying the conclusion of the Proposition.
\end{proof}

For a given function $\eta$ which is 2-balanced with respect to $P$, let $P(\eta)$ denote the polytope given by Proposition~\ref{2bal}.

\begin{prop} \label{2balsum}
If $\eta_1$, \dots, $\eta_r$ are 2-balanced with respect to $P$, then $\eta_1 + \dotsb + \eta_r$ is 2-balanced with respect to $P$ and
\[
P(\eta_1) + \dotsb + P(\eta_r) = P(\eta_1 + \dotsb + \eta_r).
\]
\end{prop}

\begin{proof}
It is immediate that $\eta_1 + \dotsb + \eta_r$ is 2-balanced, and it is easy to check that $P(\eta_1) + \dotsb + P(\eta_r)$ satisfies the conclusion of Proposition~\ref{2bal} for the function $\eta_1 + \dotsb + \eta_r$.
\end{proof}

\subsubsection{Refining to polysimplicial}

Let $V_{>0} := V \cap \bfR_{>0}$. We need the following key result:

\begin{prop} \label{decomp}
For every $(\bfZ,V)$-polytope $P \subset \bfR^n$, there exists an independent set of $\KK$-elements $\{\beta_1, \dots, \beta_r\} \subset V_{>0}$ of $V$ and rational polytopes $P_1$, \dots, $P_r \subset \bfR^n$ such that $P = \beta_1 P_1 + \dotsb + \beta_r P_r$.
\end{prop}

\begin{proof}
Let $E := E(P)$. For each $e \in E$, fix a segment $e_0 \subset \bfR^n$ such that $e$ and $e_0$ are parallel and $e_0$ has endpoints in $\KK^n$. Let $\ell(s)$ denote the length of a segment $s$. Define a function $\ell_0 : E \to \bfR$ by
\[
\ell_0(e) = {\ell(e)}/{\ell(e_0)}.
\]
Then $\ell_0(e) \in V$ for all $e \in E$. Let $\mathcal{L} = \{ \ell_0(e) : e \in E \}$.

Note that $\mathcal{L} \subset V_{>0}$. Since $V_{>0}$ is an open half-space of the $\KK$-vector space $V$ and $\mathcal{L}$ is finite, we can find a $\KK$-basis $\{\beta_1, \dots, \beta_r\} \subset V_{>0}$ of $V$ such that $\mathcal{L}$ is contained in the positive $\KK$-linear span of $\{\beta_1, \dotsc, \beta_r\}$. For each $e \in E$, let $c_{1}(e)$, $c_{2}(e)$, \dots, $c_{r}(e) \in \bf Q$ be the unique positive rational numbers such that
\[
\ell_0(e) = c_{1}(e) \beta_1 + \dotsb + c_r(e) \beta_r.
\]

Now, for each $i = 1$, \dots, $r$, define $\eta_i : E \to \bf R$ by $\eta_i(e) = c_i(e) \ell(e_0)$. We claim that each $\eta_i$ is 2-balanced with respect to $P$. Indeed, suppose $F$ is a 2-face of $P$ with vertices $v_1$, \dots, $v_k$ and edges $v_1v_2$, $v_2v_3$, \dots, $v_kv_1$. For each $i$, let
\begin{align}
\epsilon_i &:= \eta_i(v_1v_2) \frac{v_2-v_1}{\lVert v_2-v_1 \rVert} + \eta_i(v_2v_3) \frac{v_3-v_2}{\lVert v_3-v_2 \rVert} + \dotsb + \eta_i(v_kv_1) \frac{v_1-v_k}{\lVert v_1-v_k \rVert} \nonumber \\
&= c_i(v_1v_2) \overrightarrow{(v_1v_2)_0} + c_i(v_2v_3) \overrightarrow{(v_2v_3)_0} + \dotsb + c_i(v_kv_1)  \overrightarrow{(v_kv_i)_0} \label{eq:sum}
\end{align}
where $\overrightarrow{(uv)_0}$ denotes the segment $(uv)_0$ oriented in the direction $v-u$. Since each term on the right hand side of \eqref{eq:sum} is in $\KK^n$, we have $\epsilon_i \in \KK^n$ for all $i$.
Now,
\begin{align*}
\sum_{i=1}^r \beta_i \epsilon_i &= \left( \sum_{i=1}^r \beta_i c_i(v_1v_2) \right) \overrightarrow{(v_1v_2)_0} + \dotsb + \left( \sum_{i=1}^r \beta_i c_i(v_kv_1) \right) \overrightarrow{(v_kv_i)_0} \\
&= \ell_0(v_1v_2) \overrightarrow{(v_1v_2)_0} + \ell_0(v_2v_3) \overrightarrow{(v_2v_3)_0} + \dotsb + \ell_0(v_kv_1) \overrightarrow{(v_kv_i)_0} \\
&= (v_2-v_1) + (v_3-v_2) + \dotsb + (v_1-v_k) \\
&= \vec{0}.
\end{align*}
On the other hand, since $\epsilon_i \in \KK^n$ for all $i$ and $\beta_1$, \dots, $\beta_r$ are linearly independent over $\bf Q$, this implies $\epsilon_i = \vec{0}$ for all $i$. Thus each $\eta_i$ is 2-balanced.

Now, from the proof of Proposition~\ref{2bal}, the polytopes $P(\eta_1)$, \dots, $P(\eta_r)$ can be translated to have vertices in $\KK^n$ and hence are $\KK$-polytopes. By Proposition~\ref{2balsum}, we have $\beta_1 P(\eta_1) + \dotsb + \beta_r P(\eta_r) = P(\eta)$, where $\eta = \beta_1\eta_1 + \dotsb + \beta_r\eta_r$. Note that $\eta(e) = \ell_0(e)\ell(e_0) = \ell(e)$ for all $e \in E$. Thus, by the uniqueness part of Proposition~\ref{2bal}, it follows that $P(\eta) = P$, as desired.
\end{proof}

\begin{proof}[Proof of Lemma~\ref{polysimp}]
Let $P = \beta_1 P_1 + \dotsb + \beta_r P_r$ be as in Proposition~\ref{decomp}. Consider a rational triangulation of $\Cay(P_{[r]})$ that uses the vertices of each summand. This induces a fine mixed subdivision of
$P_1 + \dotsb +  P_r$ into rational polysimplices. Dilating each summand of a Cayley polytope with a positive real does not change the facet slopes of the subdivision, and yields the desired polysimplicial subdivision of $P$ as a fine mixed subdivision.
\end{proof}

\subsection{From polysimplicial to polystable}

When proving the existence of polysimplicial subdivisions, we observed that polysimpliciality of the subdivision is reduced to rationality of the triangulation of the associated Cayley polytope. Similarly, providing a polystable subdivision of a $(\bfZ,V)$-polytope $P$ is true if, given a presentation of $P$ as Minkowski sum of dilated $(\bfZ,\bfZ)$-polytopes $P_1,\dots,P_r$, we can find a unimodular triangulation of the Cayley polytope of this family, we we would be allowed to dilate each summand with a positive integer. Recall: A triangulation of a lattice polytope is unimodular if it is a triangulation into lattice simplices, such that the vertices of each maximal simplex affinely generate the lattice.

We hit an obstacle:

\subsubsection{Roadblocks}
Now, we are almost ready to prove Theorem~\ref{polysubdv} as well, but are hit with a roadblock that Proposition~\ref{decomp} put us in:
{Decomposition of $P$ as Minkowski sum plays a crucial role in our approach, but a wrong choice of such a decomposition makes further polystable subdivision impossible.}
%We realized that a decomposition of $P$ as Minkowski sums is crucial, but that hits a roadblock for the polystable case.

\begin{exam}
{Let $a,b>0$ be two real numbers linearly independent over $\bfQ$ and let $V$ be their $\bfQ$-span in $\bfR$.}  Consider the rectangle $P$ obtained as the Minkowski sum of the segments $P_1=[(0,0),(a,a)]$ and $P_2=[(0,0),(b,-b)]$. Then $P$ is a $(\bfZ,V)$-polytope, and the desired polystable subdivision cannot arise from a mixed subdivision, and in particular not from a unimodular triangulation of the underlying Cayley polytope.
\end{exam}

\subsubsection{A generalization of Knudsen-Mumford-Waterman to Cayley polytopes} 
 Indeed, at this point, we would like to see a unimodular triangulation of the underlying Cayley polytope of $P_1$ and $P_2$, or at least of some of their integer multiples. The reason for this failure, exemplified above, lies in the fact that the lattices generated by the vertices of $P_1$ and $P_2$, respectively, form subgroups of infinite order in the ambient lattice $\bfZ^2$. However, we can fix this if we restrict to Minkowski sums of co-compact $(\bfZ,V)$-polytopes, where we call polytopes co-compact if they are full-dimensional. We indeed can restrict further for our purposes.

While seemingly innocent, it turns out that the key lemma is to triangulate Cayley polytopes without loss of index. Consider $(P_{[r]})=P_1, \dots, P_r$ a family of lattice polysimplices \emph{Minkowski dominated} by a polystable $d$-polysimplex $P$, that is, the normal fan of each member is refined by the normal fan of $P$). By passing to an affine span, the general case reduces to the case in which the vertices of $P$ affinely generate the lattice.

\begin{lem}\label{decompCayley}
$\Cay(P_{[r]})$ has a unimodular triangulation.
\end{lem}

\begin{proof}
We prove this statement in several steps of increasing generality.
First, choose a distinguished vertex $v_0=0$ of $P$, and start from it a maximal path of affinely independent vertices connected by edges $(e_k)_{k\in [d]}$. Then the polystable parallelpiped obtained as the Minkowski sum $\sum e_k$ has a unimodular triangulation $T$ by the simplices
\[\conv(v_0,v_0+e_{\pi(1)}, \dots,\ v_0 + \sum_{k=1}^d e_{\pi(k)} )\]
where $\pi$ ranges over the permutations of $[d]$.
It is not hard to check that the restriction of $T$ to the support of $P$ induces a unimodular triangulation of $P$.

We now show that each summand $P_i$ can be unimodularly triangulated. Let $P = \sum_j \sigma_j$ where the $\sigma_j$ are rational simplices. With this, $P_i = \sum_j \lambda_{ij} \sigma_j$ where the $\lambda_{ij}$ are nonnegative integers (this is Minkowski equivalence). For each $k \in [d]$, set
\[
c_{ik} := \lambda_{ij} \ \text{where $j$ is such that $e_k$ is parallel to the affine span of $\sigma_j$}.
\]
By tesselating $T$ (or a face of $T$), we obtain a unimodular triangulation of $\sum c_{ik} [0,e_k]$, where $[0,e_k]$ denotes the segment from the origin to $e_k$. As before, this induces a triangulation of $\sum_j \lambda_{ij} \sigma_j = P_i$.

We next prove the Lemma when $P_i = cP$ for all $i \in [r]$ and some $c$. The above unimodular triangulation of a summand $P_i$ into unimodular simplices $\Delta$ extends to a decomposition of $\Cay(P_{[r]})$ into Cayley polytopes
$\Cay(\Delta_{[r]})$ (where $\Delta$ is taken $r$ times). Each of these Cayley polytopes is isomorphic to the polysimplex $\Delta \times \Delta^{r-1}$, and hence can be unimodularly triangulated as above.

Finally, assume that the $P_i$ are all smaller (i.e.\ contained in) than $cP$ for some $c$. Write $P_i = \sum_j \lambda_{ij} \sigma_j$ as above, and thus $\lambda_{ij} \le c$ for all $i$, $j$. Define $c_{ik}$ as before. We define a map $\phi_i : cP \to P_i$ as follows. Each point $u \in \sum ce_k$ can be uniquely written as
\[
u = v_0 + \alpha_1 e_1 + \dotsb + \alpha_d e_d
\]
where $0 \le \alpha_k \le c$ for all $\alpha_k$. We then define
\[
\phi_i(u) = v_0 + \min(\alpha_1,c_{i1}) e_1 + \dotsb + \min(\alpha_d,c_{id}) e_d
\]
which gives a point in $\sum c_{ik}e_k$. This restricts to a map $\phi_i : cP \to P_i$.

Now, given a lattice polytope of the form $\tau = \Cay(\tau_1,\dots,\tau_r)$ where $\tau_i \subset cP$ for all $i$, we define
\[
\phi(\tau) = \Cay( \phi_1(\tau_1), \dots, \phi_r(\tau_r) ).
\]
Note that $\phi(\Cay((cP)_{[r]})) = \Cay(P_{[r]})$. It is straightforward to check that for the unimodular triangulation of $\Cay((cP)_{[r]})$ described above, the image of each simplex of this triangulation under $\phi$ is again unimodular (but may degenerate), and together the images form a unimodular triangulation of $\Cay((cP)_{[r]})$. This completes the proof.
\end{proof}

%Now, if we change one of the summands $P_i$ to a smaller one, then we do so by iteratively projecting the lattice points along the vectors $(e_j)$ in their order from $1$ to $d$, to the affine span of the unique facet containing
%\[\sum_{j=1}^k e_j \]
%but not
%$\sum_{j=1}^{k-1} e_j$, and keep the combinatorics intact. The resulting simplices will degenerate or remain unimodular. Repeating this procedure with all summands yields the desired.

We call the triangulation obtained a \emph{compressed lexicographic} triangulation (short \emph{{c-lex}}) of the Cayley polytope in question (induced by the linear order as described). Note that by restricing to one of the facets of $P$, we pass to an associated Cayley polytope of that facet whose triangulation is {c-lex} as well.

Next, we show that if a polysimplex is not polystable, then we can improve the index as well, provided that the summands are all co-compact. The proof of the lemma follows ideas of \cite{Santos}, specifically their lemma 4.16. For a polyhedral complex $X$ and a face $F$, we call $\st_F X$ the star of $F$ in $X$, i.e. the minimal complex containing all faces containing $F$.

If $(c_{[r]})=(c_i)_{i=1,\dots,r}$ is a vector of $r$ positive integers, then we can consider the family $(c_{[r]})(P_{[r]})=:(c_iP_i)_{i=1,\dots,r}$. Again, we argue by increasing generality, but separate the crucial case out for convenience.

For the following, we fix the ambient lattice $\bfZ^d$ for clarity of reference.

\begin{lem}\label{decompint}
Consider a family of lattice polysimplices $(P_{[r]})$ all of which are co-compact and Minkowski dominated by a polysimplex $\widetilde P$, such that all facets of $\widetilde P$ are polystable (in the lattices of their affine spans), but $\widetilde P$ is not; its index is $N>1$.
Then there exists a distortion factor $(\widetilde{c}_{[r]})$ such that for every entrywise multiple ${c}_{[r]}$, there is a triangulation of $\Cay({c}_{[r]}P_{[r]})$ into simplices of strictly smaller index.

Moreover, the triangulation of $\Cay( cP_{[r]})$ is {c-lex} when restricted to the Cayley complex, and the refinement can be chosen to be regular.
\end{lem}

Let us explain the terminology. For a lattice (poly)simplex, the index is the index of the subgroup generated by its within the ambient lattice. The goal of the lemma is to establish that while the polytope $\Cay(P_{[r]})$ may not have a unimodular triangulation, the polytope $\Cay({c}_{[r]}P_{[r]})$, where each polytope is enlarged by a suitably large integer factor $c_i,\ i\in [r]$, does admit at least a triangulation (with vertices among the lattice points of $Cay({c}_{[r]}P_{[r]}) \cap \bfZ^d$) with better index than the index of the starting polysimplex. The index of a triangulation is measured as the maximum over the indices over its elements.

\begin{proof}
Let $\widetilde{P}=\sum_1^k \sigma_j$ where the $\sigma_j$ are affinely independent simplices, and let $\Lambda$ be the lattice generated by $\widetilde{P}$. Let $c = \prod_j (\dim \sigma_j + 1)$. With this choice, the index of $\Lambda$ in the ambient lattice $\bfZ^d$ is $N$.

%Observe that for some $c_0\le \dim \widetilde{P}$, $c_0\widetilde{P}$ has a triangulation into simplices of strictly smaller index than $N$, each adjoining an interior latticepoint $\mathbf{m}$ to the lattice $\Lambda$.
%In particular, so does the Cayley-polytope of the $r$-fold copy of $c_0\widetilde{P}$.

Consider the matrix $(\lambda_{ij})_{i\in [r], j\in [k]}$ of non-negative integers. We denote by $P[(\lambda_{ij})]$ the Cayley polytope of the $r$ summands $\sum_1^k c \lambda_{ij}\sigma_j$. It suffices to give a triangulation of $P[(\lambda_{ij})]$ with index lower than $N$. Since every $P_i$ is co-compact, we may assume all the $(\lambda_{ij})$ are positive.

%First assume that $(\lambda_{(ij)})$ is a $(0,1)$-matrix such that
%\[
%[k] = \supp(\lambda_{(1j)}) \supseteq \supp(\lambda_{(2j)}) \supseteq \dotsb \supseteq \supp(\lambda_{(rj)})
%\]
%where $\supp(\lambda)$ denotes the set of indices at which $\lambda$ is nonzero. Then $P[(\lambda_{(ij)})]$ has a unimodular triangulation in $\langle\Lambda+\mathbf{m}\rangle$ with {c-lex} boundary, as follows. Let $s$ be the largest integer such that $\supp(\lambda_{(sj)}) = [k]$. For each facet $F$ of $\tilde P$ and each fixed $i_0 \in [r]$, let $F[(\lambda_{(i_0j)})]$ be the face of of $P[(\lambda_{(i_0j)})]$ with respect to the supporting hyperplane of $F$. Then the set
%\begin{multline}
%\bigcup_{\substack{F \text{ is a} \\ \text{facet of } \tilde{P}}} \Cay( \conv(F[(\lambda_{(1j)})], \mathbf{m}), \conv(F[(\lambda_{(2j)})], \mathbf{m}), \dots, \conv(F[(\lambda_{(sj)})],\mathbf{m}),
%\\ F[(\lambda_{(s+1,j)})], F[(\lambda_{(s+2,j)})], \dots, F[(\lambda_{(r,j)})])
%\end{multline}
%is the set of maximal cells in a subdivision of $P[(\lambda_{(ij)})]$. Each cell can then be triangulated unimodularly in $\langle\Lambda+\mathbf{m}\rangle$ through Lemma~\ref{decompint}, so that the boundaries of all cells are c-lex.

Now, we want to decrease the index of $\Lambda$ in $\bfZ^d$. For this, we want to make sure that we triangulate the lattice Cayley polytope in such a way that all simplices of the triangulation generate the lattice span $\langle\Lambda , \mathbf{m} \rangle$ of $\Lambda$ and  $\mathbf{m}$ in $\bfZ^d$, where $\mathbf{m}$ is a suitable representative of a nonzero element of $\bfZ^d / \Lambda$. There is a unique minimal (in the product order) tuple $(c_{[k]})$ of positive integers such that $\sum_1^k c_j \sigma_j$ contains such a representative $\mathbf{m}$. This tuple satisfies $1 \le c_j \le \dim(\sigma_j)$ for all $j$.

Fix $i$, $j$. Recall that $c$ is divisible by $\dim \sigma_j + 1$. Let $\{\alpha_s\}_s$ range over all multiples of $\dim \sigma_j + 1$ from 0 to $c\lambda_{ij}$ inclusive. We define $A_s$ to be the translation of the polytope
\[ %\label{eq:conc1}
\alpha_s\sigma_j + \sum_{j' \in [k] \setminus \{j\}} c\lambda_{ij'}\sigma_{j'}
\]
such that $A_s$ has the same barycenter as $P_i$. (Note that for $\alpha_s = c\lambda_{ij}$, we have $A_s = P_i$.)
In addition, let $\beta_s = \alpha_s - c_j$ for $\alpha_s$ from $\dim \sigma_j + 1$ to $c\lambda_{ij}$ inclusive. Define $B_s$ to be the translation of the polytope
\[ %\label{eq:conc2}
\beta_s \sigma_j+\sum_{j' \in [k] \setminus \{j\}} (c\lambda_{ij'}-c_{j'})\sigma_{j'}
\]
such that the vertices of $B_s$ are contained in $(A_s \setminus A_{s-1}) \cap (\Lambda + \mathbf{m})$. This expression is well-defined because $c\lambda_{ij'} > c_{j'}$ for all $j'$.

Given a polytope $Q$ that is Minkowski dominated by $\widetilde{P}$, for each face $F$ of $\widetilde{P}$, there is a unique maximal face of $Q$ whose set of normal vectors contains the set of normal vectors of $F$. We denote this face by $Q^F$. Now, we observe that $P_i$ can be subdivided into polytopes with vertices in $\langle\Lambda , \mathbf{m} \rangle$ such that, when these polytopes are viewed as lattice polytopes of $\langle\Lambda , \mathbf{m} \rangle$, each polytope is isomorphic to one of
\begin{align*}
&\Cay(A_{s}^F, B_{s}^F)  \\  
&\Cay(A_{s-1}^F, B_{s}^F) \\ 
&\Cay(A_s^F, A_{s-1}^F, B_s^F)
\end{align*}
for some face $F$ of $\widetilde{P}$. (In fact, we can describe the $F$ explicitly: In the first two expressions, $F$ ranges over the set $\mathcal{S}$ of faces of $\widetilde{P}$ which contain a proper face of $\sigma_j$ as a Minkowski summand, and in the third expression, $F$ ranges over all proper faces of elements in $\mathcal{S}$.) These polytopes can then be unimodularly triangulated in $\langle\Lambda , \mathbf{m} \rangle$ by Lemma~\ref{decompCayley}. We thus have a triangulation of $P_i$ with index lower than $N$.

Finally, we triangulate $P[(\lambda_{ij})]$. First suppose $\lambda_{ij} = 1$ for all $i$, $j$. From what we just showed, each $P_i$ can be triangulated with index lower than $N$, and this triangulation is the same for all $P_i$. Then as in the proof of Lemma~\ref{decompCayley}, we can extend this to a triangulation of $P[(\lambda_{ij})]$ with index lower than $N$. Now suppose we have such a triangulation of $P[(\lambda_{ij})]$, and we increase $\lambda_{ij}$ by 1. Then the new Cayley polytope can be subdivided using the simplices of the original triangulation along with polytopes isomorphic to
\begin{align*}
&\Cay(A_{s}^F, B_{s}^F, P_1^F, \dots)  \\  
&\Cay(A_{s-1}^F, B_{s}^F, P_1^F, \dots) \\ 
&\Cay(A_s^F, A_{s-1}^F, B_s^F, P_1^F, \dots)
\end{align*}
where $A_s$, $B_s$, and $F$ are as before, and the dots in the above expressions indicate all $P_{i'}^F$ for $i' \neq i$. These Cayley polytopes can be triangulated unimodularly in $\langle\Lambda ,\mathbf{m} \rangle$ by Lemma~\ref{decompCayley}, which completes the proof. Regularity of the ensuing subdivision is a straightforward verification.
\end{proof}

%Next, assume we modify some entry, say $\lambda_{(ab)}$, from $1$ to some positive integer, say $\ell$. Then
%$\ell \sigma_b+\sum_{j \in [k]\setminus\{b\}} c_0\lambda_{aj}\sigma_j$ is decomposed into $\lfloor\frac{\ell}{2}\rfloor$ concentric translated copies of \[\alpha_k\sigma_b+\sum_{j \in [k]\setminus\{b\}} c_0\lambda_{aj}\sigma_j,\]
%where $\alpha_k$ goes over nonnegative even integers between $0$ and $\ell$
%(if $\ell$ is even) or odd integers (if $\ell$ is odd),
%with with vertices in $\Lambda$. These concentric copies are interlaced by translations of
%\[\beta_k\sigma_b+\sum_{j \in [k]\setminus\{b\}} c_0(\lambda_{aj}-1)\sigma_j,\]
%where $\beta_k$ ranges over the nonnegative odd integers (if $\ell$ is even) or even integers (if $\ell$ is even) between $0$ and $\ell$. The space between these interlaced copies is triangulated unimodularly in $\langle \Lambda + \mathbf{m} \rangle$ by Lemma~\ref{decompCayley}. This triangulation extends to the Cayley polytope by the same lemma.

\subsubsection{Proof of polystability}

We need to slightly modify the previous lemma: If the Cayley polytope of Lemma~\ref{decompint} is part of a larger complex, then we wish to not move farther away from polystability eslewhere.

\begin{lem}\label{decomp2}
Consider a Cayley triangulation $X$ of an $r$-fold family, a Cayley polytope of some dimension (say $d$) decomposed into Cayley polytopes that whose $r$ summands are Minkowski dominated by lattice polysimplices. Consider also a linear order on its cardinality $r$ Cayley simplices. Consider $F$ a non-polystable facet of $X$ (i.e. its index $j$ is larger than one in its affine span). Assume that all of the summands of $F$ are co-compact in the affine span of $F$.

Then there exists a distortion factor $(\widetilde{c}_{[r]})$ such that for every entrywise multiple ${c}_{[r]}$, there is a subdivision of the family ${c}_{[r]}X$ into Cayley polytopes of lattice polysimplices such that each facet of $S=\st_F X$ is subdivided into Cayley polytopes of strictly smaller index, and all other facets have unchanged index.

Moreover, the triangulation of $\partial \Cay( cS_{[r]})$ is {c-lex} (with respect to the linear order) and the refinement can be chosen to be regular.
\end{lem}

\begin{proof}
A resizing of $X$ can be refined via Lemma~\ref{decompCayley} outside of $\st_F X$. We therefore only have to discuss the situation in the star of $F$. For this, we can use the argument of the previous lemma, which we extend to facets of $X$ containing $F$.
\end{proof}

We are now ready to prove our main theorem.

\begin{proof}[Proof of Theorem~\ref{polysubdv}]
Following Lemma~\ref{polysimp}, we can assume that $P$ has a mixed subdivision into polysimplicial tiles arising from unimodular decompositions of Minkowski summands $(P_i)$. Unfortunately, corresponding summands may degenerate.

Hence, we change the summands by changing the independent set $\beta_1, \dots, \beta_r$ in $V$ to a generic small linear transformation $\beta'_i$ (small enough so as not violate the positivity condition in the proof of Lemma~\ref{polysimp}). We then obtain a representation of the polysimplicial subdivision of $P$ as a mixed subdivision arising from combinatorially equivalent subdivisions of polytopes $(P_i')$ generated from this basis. The individual facets arise therefore sections of Cayley polytopes of co-compact Minkowski-equivalent polysimplices in the $P_i'$, to which we can apply Lemma~\ref{decomp2} to reduce the index of some of the lattices generated by Cayley polytopes while leaving the others the same. We then change $\beta_1, \dots, \beta_r$ again and repeat until the index of each Cayley polytope is one.
\end{proof}

\subsection{Polystable refinement, continued}

We now turn to refinements of Theorem~\ref{polysubdv}, the most important of which is the generalization to complexes.

\subsubsection{Polystable refinement of complexes}

Given two polyhedral complexes $X$ and $X'$, we say that $X'$ \emph{refines} $X$ if every cell of $X'$ is contained in a cell of $X$. We prove the following result.

\begin{theor} \label{compsubdv}
Let $X$ be a polyhedral complex where every cell is a $(\bfZ,V)$-polytope. Then there exists a polyhedral complex $X'$ which refines $X$ such that every cell of $X'$ is a $(\bfZ,V)$-polytope and a product of semistable simplices.
\end{theor}

\begin{proof}
Let $L$ be the set of all edge lengths of all edges of $X$. As in the proof of Proposition~\ref{decomp}, we can choose a $\KK$-basis $\{\beta_1,\dotsc,\beta_r\} \subset V_{>0}$ of $V$ such that $L$ is contained in the positive span of $\{\beta_1,\dotsc,\beta_r\}$. Also as in that proof, for each cell $C$ of $X$, there are 2-balanced functions $\eta_1^C$, \dots, $\eta_r^C$ with respect to $C$ such that $C = \beta_1C(\eta_1^C) + \dotsb + \beta_rC(\eta_r^C)$, where each $C(\eta_i^C)$ is a $\KK$-polytope. We may assume that each of these summands is subdivided in a regular unimodular triangulation as the Knudsen-Mumford-Waterman construction extends to triangulations of polyhedral complexes.

It is clear from the construction of the $\eta_i^C$ that when two cells $C$, $C'$ share an edge $e$, we have $\eta_i^{C}(e) = \eta_i^{C'}(e)$ for all $i$. Thus, for all cells $C$ and $C'$ which share a face $F$ and for all $i$, we have that $C(\eta_i^{C})$ and $C'(\eta_i^{C'})$ share the common face (up to translation) $F(\eta_i^F)$, where $\eta_i^F$ is the common restriction of $\eta_i^C$ and $\eta_i^{C'}$ to the edges of $F$.

Now, suppose we have a function $f : S(X) \to \bf R$. For each cell $C$ of $X$ and for all $i = 1$, \dots, $r$, let $\psi_i^C$ be the isomorphism from $G(\beta_i C(\eta_i^C))$ to $G(C)$ given by Proposition~\ref{2bal}. Then for generic $f$, the functions $f \circ \psi_1^C$, \dots, $f \circ \psi_r^C$ induce a fine mixed subdivision $X_C$ of $C$. As in the proof of Lemma~\ref{polysimp}, the cells of $X_C$ are all $(\bfZ,V)$-polytopes and products of simplices. Moreover, for each face $F$ of $C$, the subdivision of $F$ given by $X_C$ is induced by $f \circ \psi_1^F$, \dots, $f \circ \psi_r^F$, where $\psi_i^F$ is the isomorphism from $G(\beta_i F(\eta_i^F))$ to $G(F)$. Hence for all cells $C$ and $C'$ sharing a common face, we have that $X_C$ and $X_{C'}$ agree on that face. Thus the collection $\bigcup X_C$, where the union ranges over all cells $C$ of $X$, gives a polysimplicial subdivision. We can now transform it into a polystable subdivision using Lemma~\ref{decomp2} as before.
\end{proof}

\subsubsection{Regular refinement}

Recall the definition of regular subdivision from Section~\ref{sec:regsubdv}. In the case where $X$ is a subdivision of a polytope, we can guarantee the refinement in Theorem~\ref{compsubdv} to be regular:

\begin{theor} \label{regref}
Let $X$ be a $(\bfZ,V)$-subdivision of a polytope $P$. Then there exists a regular polystable subdivision $X'$ which refines $X$. Moreover, there is a function $f$ which induces $X'$ such that $P^f$ is a $(\bfZ,V)$-polytope.
\end{theor}

The proof follows from the following three propositions.

\begin{prop}
For every $(\bfZ,V)$-subdivision $X$ of a polytope $P$, there is a regular $(\bfZ,V)$-subdivision which refines $X$.
\end{prop}

\begin{proof}
Without loss of generality, we may assume $P$ is full-dimensional in $\bfR^n$. Let $\mathcal A$ be the collection of all affine hyperplanes spanned by facets of cells of $X$. Then $\mathcal A$ induces a subdivision $X_{\mathcal A}$ of $P$ given by the collection of closures of connected components of $P - \bigcup_{H \in \mathcal A} H$. Clearly this is a $(\bfZ,V)$-subdivision, and every subdivision induced in this way by a collection of hyperplanes is regular.
\end{proof}

\begin{prop}
For every regular $(\bfZ,V)$-subdivision $X$ of a polytope $P$, there is a regular polystable subdivision which refines $X$.
\end{prop}

\begin{proof}
By the proof of Theorem~\ref{compsubdv}, there is a polystable refinement $X'$ of $X$ such that for every cell $C$ of $X$, the collection of cells of $X'$ contained in $C$ form a regular subdivision $X_C$ of $C$. For each cell $C$ of $X$, let $f_{C} : S(X_C) \to \bfR$ be a function which induces $X_C$. Additionally, we can choose the $f_{C}$ such that if two subdivisions $X_C$ and $X_{C'}$ share a vertex $v$, then $f_{C}(v) = f_{C'}(v)$. Hence, there is a function $f' : S(X') \to \bf R$ such that for all $C$, we have that $f'$ restricted to $S(X_C)$ is $f_C$.

Let $f_0$ be a function which induces $X$. Let $f : S(X') \to \bfR$ be the function which restricts to $f_0$ on the vertices of $X$ and equals 0 otherwise. Then for small enough $\epsilon > 0$, $f + \epsilon f'$ induces $X'$. Thus $X'$ is regular.
\end{proof}

\begin{prop}
For every regular $(\bfZ,V)$-subdivision $X$ of a polytope $P$, there is a function $f$ which induces $X$ such that $P^f$ is a $(\bfZ,V)$-polytope.
\end{prop}

\begin{proof}
As in the proof of Theorem~\ref{compsubdv}, there is a $\KK$-basis $\{\beta_1,\dotsc,\beta_r\} \subset V_{>0}$ of $V$ satisfying the following: For each cell $C$ of $X$, there are 2-balanced functions $\eta_1^C$,~\dots,~$\eta_r^C$ with respect to $C$ such that $C = \beta_1C(\eta_1^C) + \dotsb + \beta_rC(\eta_r^C)$ and each $C(\eta_i^C)$ is a $\KK$-polytope. Moreover, the proof implies the following:
\begin{compactenum}[(i)]
\item For each $i=1$,\dots,$r$ there are functions $\eta_i: E(X) \to \bfR_{>0}$ such that for each cell $C$ of $X$, $\eta_i$ restricted $E(C)$ is $\eta_i^C$.
\item Let $\eta_i^P$ be the restriction of $\eta_i$ to $E(P)$. Then the collection $X_i := \{ C(\eta_i^C) \}_{C \in X}$ forms (after translation) a subdivision of $P_i := P(\eta_i^P)$.
\end{compactenum}
We now prove the following two lemmas.

\begin{lem}
For each $i$, $X_i$ is a regular subdivision of $P_i$.
\end{lem}

\begin{proof}
Let $f$ be a function which induces $X$. Let $E^f$ denote the set of lower edges of $P^f$. (A lower edge is an edge which is a lower face, as defined in Section~\ref{sec:regsubdv}.) Fix $i$, and define a function $\eta_i^f : E^f \to \bfR_{>0}$ by
\[
\eta_i^f(e) = \eta_i(\pi(e)) \frac{\ell(e)}{\ell(\pi(e))},
\]
where $\pi$ is defined as in Section~\ref{sec:regsubdv}. Now, suppose $F$ is a 2-dimensional lower face of $P^f$ with vertices $v_1$, \dots, $v_k$ and edges $v_1v_2$, \dots, $v_kv_1$. Let
\[
\epsilon := \eta_i^f(v_1v_2) \frac{v_2-v_1}{\lVert v_2-v_1 \rVert} + \eta_i^f(v_2v_3) \frac{v_3-v_2}{\lVert v_3-v_2 \rVert} + \dotsb + \eta_i^f(v_kv_1) \frac{v_1-v_k}{\lVert v_1-v_k \rVert}.
\]
We claim that $\epsilon = \vec{0}$. Indeed, we have
\begin{align*}
\pi(\epsilon) &= \eta_i(\pi(v_1v_2)) \frac{\pi(v_2)-\pi(v_1)}{\lVert \pi(v_2)-\pi(v_1) \rVert} + \dotsb + \eta_i(\pi(v_kv_1)) \frac{\pi(v_1)-\pi(v_k)}{\lVert \pi(v_1)-\pi(v_k) \rVert} \\
&= 0
\end{align*}
because $\pi(F)$ is a 2-face of $X$ and $\eta_i$ is 2-balanced. On the other hand, $\epsilon$ is parallel to $F$, and $F$ lies in a hyperplane $\{x \in \bfR^{n+1} : \phi(x) = c\}$ with $\phi(0,\dots,0,1) > 0$. Hence, $\pi(\epsilon) = 0$ implies $\epsilon = 0$, as desired.

Thus, using the proof of Proposition~\ref{2bal}, we can construct a polytope $P'$ with vertex set $S'$, lower edge set $E'$, and lower graph $G' = (S',E')$ such that there is a graph isomorphism $\psi : G^f \to G'$, where $G^f$ is the lower graph of $P^f$, and
\[
\psi(v) - \psi(u) = \eta_i^f(uv) \frac{v-u}{\lVert v-u \rVert}
\]
for every lower edge $uv$ of $P^f$. The collection of projections of lower facets of $P'$ under $\pi$ is precisely the subdivision $X_i$, as desired.
\end{proof}

By construction, the subdivisions $X_i$ can be translated to be $\KK$-subdivisions. We then have the following.

\begin{lem}
For each $i$, there is a function $f_i$ which induces $X_i$ such that $P_i^{f_i}$ is a $\KK$-polytope.
\end{lem}

\begin{proof}
For any function $f: S(X_i) \to \bf R$, the condition that $f$ induces $X_i$ can be expressed as a system of linear equations and inequalities on its values: Namely, that for every ($n$-dimensional) cell $C$ of $X_i$, we have that $C^f$ spans an $n$-dimensional affine subspace of $\bfR^{n+1}$, and for every vertex $v$ of $X_i$ not in $C$, $(v,f(v))$ lies strictly above this subspace (where ``above'' means in the (0,...,0,1) direction). Since $X_i$ is regular, this system has a solution. Moreover, since $X_i$ is a $\KK$-subdivision, all of the linear equations can be taken to have rational coefficients, and hence the system has a rational solution $f_i$. Then $P_i^{f_i}$ is a $\KK$-polytope, as desired.
\end{proof}

We can now complete the proof. Let $f_1$, \dots, $f_r$ be as in the previous Lemma. Then $\beta_1 P_1^{f_1}  + \dotsb + \beta_r P_r^{f_r}$ is a $(\bfZ,V)$-polytope, and the collection of projections of lower facets of this polytope under $\pi$ is precisely $X$.
\end{proof}

\subsubsection{Maps}

Recall that a $(\bfZ,V)$-affine map is a map $\bfR^n \to \bfR^m$ of the form $x \mapsto Ax + v$, where $A \in S^{m \times n}$ and $v \in L$. A morphism of $(\bfZ,V)$-polytopes $P$, $Q$ is a map $P \to Q$ induced by a $(\bfZ,V)$-affine map. If $X$ and $Y$ are polyhedral subdivisions, a map $X \to Y$ is a \emph{polyhedral map} if the image of every cell of $X$ is contained in a cell of $Y$.

\begin{theor}
Let $P$, $Q$ be $(\bfZ,V)$-polytopes, $m : P \to Q$ a morphism of $(\bfZ,V)$-polytopes, and $Y$ a $(\bfZ,V)$-subdivision of $Q$. Then there exist regular polystable subdivisions $X'$, $Y'$ of $P$ and $Q$, respectively, where $Y'$ refines $Y$, such that the induced map $m' : X' \to Y'$ is a polyhedral map.
\end{theor}

\begin{proof}
By Theorem~\ref{regref}, there is a regular polystable subdivision $Y'$ which refines $Y$. Let $X := \{ m^{-1}(C) \}_{C \in Y'}$. Then $X$ is a $(\bfZ,V)$-subdivision of $P$. By Theorem~\ref{regref}, there is a regular polystable subdivision $X'$ which refines $X$. Then $X'$ and $Y'$ satisfy the conclusion of the theorem.
\end{proof}

\section{Applications to log schemes}\label{lastsec}

\subsection{Monoidal subdivisions of log schemes}
We will work with log structures defined in a topology $\tau$, which can be Zariski, \'etale or flat, though we are mainly interested in the first two cases. We only consider quasi-coherent integral log structures (recalled below).

\subsubsection{Charts}
A (global) affine chart for the log structure $M_X$ consists of an integral monoid $P$ and a homomorphism $P\to\Gamma(\calO_X)$ such that the associated log structure is $M_X$. Equivalently, one can present a chart as a strict morphism of log schemes $X\to\Spec(\bfZ[P])$, or as a map of monoidal spaces $\pi\:(X,M_X)\to\Spec(P)$ such that the log structure associated with $\pi^*(P)$ is $M_X$. The latter approach can be conveniently globalized as follows: a {\em monoscheme chart} of $X$ is a morphism $\pi:(X,M_X)\to(Z,M_Z)$, where the target is a monoscheme and $M_X$ is the log structure associated with $\pi^*(M_Z)$. In particular, the latter notion allows to work with disjoint unions.

\subsubsection{Quasi-coherence}
Recall that a log scheme is {\em quasi-coherent} if its log structure possesses charts $\tau$-locally. This happens if and only if it possesses a {\em $\tau$-chart} as follows: a strict morphism of log schemes $X'\to X$ which is a $\tau$-covering on the level of schemes and a monoscheme chart $X'\to Z$ for $M_{X'}$.

\subsubsection{Monoidal pullbacks}
The following result defines pullbacks of morphisms of monoschemes with respect to monoscheme charts.

\begin{lem}\label{monopullback}
Assume that $X$ is a log scheme, $X\to Z$ is a monoscheme chart, and $Z'\to Z$ is a morphism of monoschemes.

\begin{compactenum}[(i)] \item There exists a universal log scheme $X'$ over $X$ such that the composed map of monoidal spaces $X'\to Z$ factors through $Z'\to Z$. We call $X'\to X$ the {\em pullback} of $Z'\to Z$ and use the symbolical notation $X'=X\times_ZZ'$.

\item The pullbacks are compatible in the following sense: if $Y\to X$ is a strict morphism of log schemes then $Y\times_ZZ'=Y\times_XX'$.
\end{compactenum}
\end{lem}
\begin{proof}
In the affine case $X=X'$, $Z=\Spec(P)$ and $Z'=\Spec(P')$, we simply set $X'=X\times_{\Spec(\bfZ[P])}\Spec(\bfZ[P'])$. Clearly, this definition satisfies the universal property (i), and (i) holds for affine strict morphisms $Y\to X$. Furthermore, the affine construction is compatible with localizations of $P$ and $P'$, hence it uniquely extends to the case when $Z$ and $Z'$ are arbitrary monoschemes, and, again, it is clear this construction satisfies (i), and the pullback compatibility holds for strict morphisms $Y\to X$.
\end{proof}

\subsubsection{Monoidal morphisms}
We say that a morphism $X'\to X$ of log schemes is {\em monoidal} if $\tau$-locally it is a monoidal pullback. Concretely, this means that there exists a $\tau$-chart $Y\to X$, $Y\to Z$ such that the base change $X'\times_XY\to Y$ is the pullback of a morphism of monoschemes $Z'\to Z$. If, in addition, $Z'\to Z$ can be chosen birational, a partial subdivision, or a subdivision, then we say that $X'\to X$ is {\em monoidally birational}, a {\em partial monoidal subdivision}, or a {\em monoidal subdivision}, respectively.

\subsubsection{Fans of log schemes}
By a {\em fan} of a log scheme $X$ we mean a morphism of monoidal spaces $\oX:=(X,\oM_{X})\to F$ such that $F$ is a fan and $f^*(M_F)=\oM_{X}$.

\begin{lem}\label{fanlem}
Let $X$ be a quasi-coherent log scheme.
\begin{compactenum}[(i)]
\item  Any monoscheme chart $X\to Z$ induces a fan $\oX\to F=\oZ$.

\item Conversely, any affine fan $\oX\to F=\oSpec(Q)$ is induced from an affine chart $X\to\Spec(P)$ with $\oP=Q$.
\end{compactenum}
\end{lem}
\begin{proof}
Choose a point $x\in X$ and let $z\in Z$ be its image. Then $M_{X,x}\to\calO_{X,x}$ is the log structure associated with the pre-log structure $M_{Z,z}\to \calO_{X,x}$. The functor from pe-log structures to log structures modifies the units, but keep the sharpenings unchanged.  Therefore, $\oM_{X,x}=\oM_{Z,z}$ and we obtain that $\oX\to\oZ$ is a fan. This proves (i).

For (ii), set $M=\Gamma(M_{X})$. Then the affine fan is determined by the homomorphism of global sections $Q\to\oM$. Setting $P=M\times_\oM Q$ we obtain a homomorphism of monoids $P\to M$ whose sharpening is $Q\to\oM$. It remains to show that the map $f\:(X,M_X)\to Z=\Spec(P)$ corresponding to $P\to M\to\Gamma(\calO_{X})$ is a chart. For this we should show that for any $x\in X$ with image $z=f(x)$ the log structure associated with $M_{Z,z}\to\calO_{X,x}$ coincides with $M_{X,x}$, and it suffices to show that the sharpening of $M_{Z,z}\to M_{X,x}$ is an isomorphism. On the level of sets, $f$ coincides with the fan $X\to F$. Hence $\ol{M_{Z,z}}=\oM_{F,z}=\oM_{X,x}$, as required.
\end{proof}

\subsubsection{Fan pullbacks}
In general, there is no natural way to associate to a morphism of fans $h\:F'\to F$ a pullback morphism of log schemes $X'\to X$, and the obstacle is in finding a canonical lifting of $h$ to a morphism of monoschemes. The following result is based on the fact that birational morphisms of fans lift to monoschemes uniquely.

\begin{theor}\label{fanpullback}
There is a unique up to unique isomorphism construction that given a fan $\oX\to F$ and a birational morphism of fans $F'\to F$ outputs a birationally monoidal morphism $X'\to X$, whose source will be symbolically denoted $X'=X\times_FF'$, so that the following compatibility conditions are satisfied:
\begin{compactenum}[(i)]
\item If $Y\to X$ is a strict morphism, then $Y\times_FF'=Y\times_XX'$.

\item If $\oX\to F$ lifts to a monoscheme chart $X\to Z$ with $\oZ=F$, and $Z'\to Z$ is the birational lift of $F'\to F$ (Corollary~\ref{bircor}) then $X'=X\times_ZZ'$.
\end{compactenum}
\end{theor}
\begin{proof}
It is easy to see that the question is local on $F$, so we can assume that $F=\oSpec(P)$. By Lemma~\ref{fanlem}(ii), the fan chart $\oX\to F$ can be lifted to a monoscheme chart $X\to Z$ with $\oZ=X$. By Lemma~\ref{birlem}(i), the morphism $F'\to F$ lifts uniquely to a birational morphism of monoschemes $Z'\to Z$, and we set $X'=X\times_ZZ'$. Moreover, comparing different monoscheme charts and using uniqueness of liftings of $F'\to F$, we obtain that $X'$ depends only on $F'\to F$ and hence can be denoted $X'=X\times_FF'$. Clearly, this construction satisfies (ii), and pullback compatibility (i) reduces to its monoscheme analog from Lemma~\ref{monopullback}(i).
\end{proof}

The lifting from fans to monoschemes preserves (partial) subdivisions by Corollary~\ref{subdivcor}. Therefore, the above proof also implies the following property of pullbacks:

\begin{lem}\label{fansublem}
Keep notation of Theorem~\ref{fanpullback} and assume that $F'\to F$ is a subdivision or a partial subdivision. Then $X\times_FF'\to X$ is a monoidal subdivison or a monoidal partial subdivision, respectively.
\end{lem}

The following remark will not be used, so we skip an easy justification.

\begin{rem}
It is easy to see that $X'=X\times_FF'$ is the universal log scheme over $X$ such that the map of monoidal spaces $\oX'\to F$ factors through $F'$. In particular, our notion of fan pullback agrees with that of Kato, and hence generalizes Kato's definition in few aspects: log schemes are only assumed to be quasi-coherent and morphism of fans are only assumed to be birational.
\end{rem}

\subsubsection{Properties of monoidal pullbacks}
Now we will check that certain properties of morphisms of fans or monoschemes are transformed to their scheme-theoretic analogs.

\begin{theor}\label{monoidbirth}
Assume that $X$ is a quasi-coherent log scheme, and we are given either a monoschemes chart $X\to F$ or a fan $\oX\to F$ of $X$. Assume, furthermore, that $f\:F'\to F$ is a birational morphism and $h\:X'=X\times_FF'\to X$ is the induced monoidally birational morphism.

\begin{compactenum}[(i)]
\item If $f$ is of finite type then $h$ is a logarithmically smooth morphism.

\item If $f$ is a subdivision (resp. a partial subdivision) then $h$ is proper (resp. separated).
\end{compactenum}
\end{theor}
\begin{proof}
Clearly, both claims are local on $F$. So, we can assume that $X'\to X$ is the pullback of a morphism of monoschemes $Z'\to\Spec(P)$. (In the case of fans, we use Lemma~\ref{fanlem}(ii) and Corollary~\ref{subdivcor} to lift fans to monoschemes.) Since $F$ is of finite type, $Z'$ is covered by open affines $\Spec(P_i)$, $1\le i\le n$, where each $P_i$ is finitely generated over $P$. Therefore, $X'\to X$ is the pullback of a morphism $Y'\to Y=\Spec(\bfZ[P])$, where $Y'$ is glued from $Y_i=\Spec(\bfZ[P_i])$, $1\le i\le n$.

The morphisms $Y_i\to Y$ are of finite type. Moreover, since $P^\gp=P_i^\gp$, these morphisms are automatically log smooth, and hence $X'\to X$ is log smooth, thereby proving (i).

Note that $Y$ and $Y'$ are integral with generic point $\Spec(K)$ for $K=\mathrm{Frac}(\bfZ[P^\gp])$, and hence in the properness or separatedness criteria for $Y'\to Y$ it suffices to consider the valuation rings $R$ with $K=\Spec(R)$. For example, see \cite[Chapter II, Exercise 4.5]{Hartshorne} or \cite[Proposition 3.2.3]{temrz} applied to $\Spec(K)\to Y'\to Y$.

Assume that $R$ is a valuation ring of $K$ such that $\Spec(K)\into Y$ factors through $\Spec(R)$. Then $R\cap P^\gp$ is a valuative monoid of $P^\gp$ containing $P$, and $P_i\subseteq R\cap P^\gp$ if and only if $\bfZ[P_i]\subseteq R$. Therefore, there is a one to one correspondence between the factorings of $\Spec(R)\to Y$ through $Y'$ and factorings of $\Spec(R\cap P^\gp)\to\Spec(P)$ through $Z'$. If $f$ is a subdivision (resp. a partial subdivision) then there exists precisely (resp. at most) one such factoring, and we obtain that $Y'\to Y$ is proper (resp. separated). This proves (ii).
\end{proof}

Given a log scheme $X$ let $X_\mathrm{tr}$ denote the locus on which the log structure is trivial in the sense that $M_X|_{X_\mathrm{tr}}=\calO_X^\times|_{X_\mathrm{tr}}$.

\begin{cor}\label{monoidbircor}
Any monoidal subdivision $f\:X'\to X$ is a proper morphism, which is an isomorphism over $X_\mathrm{tr}$. In particular, if $X_\mathrm{tr}$ is dense then $f$ is a modification.
\end{cor}
\begin{proof}
Since a morphism is proper if and only if its faithfully flat base change is proper, this follows from Theorem~\ref{monoidbirth}.
\end{proof}

\subsection{Log varieties over valuation rings}

\subsubsection{Notation}
Fix, now, a valuation ring $\calO$ and let $K$ be the fraction field. We assume that $\calO$ is of height one and the group of values $V=K^\times/\calO^\times$ is divisible. Without restriction of generality, we also fix an ordered embedding $V\into\bfR$. We provide $S=\Spec(\calO)$ with the log structure given by $R=\calO\setminus\{0\}$. Note that $\oR=V_{\ge 0}$.

\subsubsection{Log varieties}\label{logvarsec}
By a {\em log variety over $S$} we mean a quasi-coherent log scheme $X$ over $S$ such that the underlying morphism of schemes $S\to X$ is flat of finite presentation and each homomorphism of monoids $\oR\to\oM_{X,x}$ is injective and finitely generated. By default we assume that the topology $\tau$ is \'etale. If $\tau$ is Zariski then we say that $X$ is a Zariski log variety.

For any $R$-monoid $P$ set $\bfA_P:=S\otimes_{\bfZ[R]}\bfZ[P]$ with the log structure induced by $P$ over $S$. This is a Zariski log variety when $P$ is integral and $R\into P$ is of finite presentation. An arbitrary log variety $\tau$-locally admits a strict morphism to some $\bfA_P$ as above. A log variety $X$ is called {\em log smooth} if $\tau$-charts $f_i\:X_i\to\bfA_{P_i}$ can be chosen to be \'etale morphisms.

\begin{rem}
For shortness we adopt this ad hoc definition. It is easy to see that one can only require in the definition that $f_i$ are smooth.
\end{rem}

\subsubsection{Existence of fans}
In order to apply the results of \S\ref{combsec} to log varieties, we will need to use global fans. So our next goal is to provide criteria when such fans exist. First, we will recall essentially known material for fine log schemes and then transfer it to log varieties over $\calO$ by use of an approximation.

\subsubsection{Zariski subdivision}
First, recall that any fine log scheme can be monoidally modified to a Zariski one.

\begin{lem}\label{zarlem}
Let $X$ be a fine log scheme. Then there exists a monoidal subdivision $X'\to X$ such that $X'$ is a Zariski log scheme.
\end{lem}
\begin{proof}
If $X$ is fs then this is proved in \cite[Theorem~5.4]{Niziol}. It remains to use that monoidal subdivisions are preserved by compositions and the saturation $X^\sat\to X$ is a monoidal subdivision.
\end{proof}

\subsubsection{Fans of fine log schemes}
Clearly, any Zariski log scheme locally possesses a fan. Furthermore, if a noetherian Zariski log scheme $X$ possesses a global fan, then there exists an initial fan $X\to F$, and $F$ is obtained by gluing fans $\oSpec(\oM_{X,\eta_i})$, where $\eta_i$ are generic points of the logarithmic strata of $X$. For a proof see, for example, \cite[Proposition~4.7]{Ulirsch}. In general, a Zariski log scheme $X$ possesses an open covering $X=\cup X_i$ such that $X_i$ and hence also $X_{ij}=X_i\cap X_j$ possess fans. Let $X_i\to F_i$ and $X_{ij}\to F_{ij}$ be the initial fans. Note that maps $F_{ij}\to F_i$ are local isomorphisms. If the diagram $\{F_{ij}\to F_i\}_{i,j}$ possesses a colimit $F$ such that all maps $F_i\to F$ are local isomorphisms, then $F$ is a global fan of $X$. In this case we say that the colimit $F$ is {\em nice}. Global fan does not exist in general, see \cite[Example~B.1]{Gross-Siebert}.

To analyse the situation it is more illustrative to switch to the equivalent language of rational polyhedral complexes. For simplicity, let us denote them $F_i$ too. We follow \cite[\S2.6]{ACP}. Clearly, the colimit $\Sigma$ of $\{F_{ij}\to F_i\}$ exists as a topological space. Each face $S$ of $F_i$ is mapped in $\Sigma$ to a quotient $S/G_S$ by a certain group of automorphisms. If all groups $G_S$ are trivial, then $\Sigma$ acquires an induced structure of a polyhedral complex $F$, which makes it a nice colimit. (In such case, one says that the log scheme $X$ has no monodromy.) In particular, it easily follows that $B(\Sigma)$ always possesses a natural structure of a polyhedral complex, which is the nice colimit of the barycentric subdivision $\{B(F_{ij})\to B(F_i)\}_{i,j}$.

\begin{rem}
In \cite[\S2.6]{ACP} one defines the category of generalized polyhedral complexes whose elements are colimits of diagrams as above. In particular, any Zariski log scheme possesses a fan in the sense of a generalized polyhedral complex. The above result actually states that the barycentric subdivision of any generalized cone complex is a usual cone complex.
\end{rem}

\begin{lem}\label{zarlem2}
Let $X$ be a fine log scheme. Then there exists a monoidal subdivision $X'\to X$ such that $X'$ is a Zariski log scheme possessing a global fan.
\end{lem}
\begin{proof}
By Lemma~\ref{zarlem} we can assume that $X$ is Zariski. Let us prove that the barycentric monoidal subdivision $X'$ of $X$ possesses a global fan, as required. Choose an open covering $X=\cup X_i$ such that each $X_i$ possesses a fan, and let $F_i$ and $F_{ij}$ be the initial fans of $X_i$ and $X_{ij}=X_i\cap X_j$. Essentially by the definition, the barycentric subdivisions $B(F_i)$ and $B(F_{ij})$ are fans of $X_i$ and $X_{ij}$ in $X'$, and it remains to recall that the diagram $\{B(F_{ij})\to B(F_i)\}_{i,j}$ possesses a nice colimit by the discussion above.
\end{proof}

\subsubsection{Approximation}
The construction of a subdivision $X'\to X$ in \cite[Theorem~5.4]{Niziol} is canonical, and it goes by blowing up ideals generated by certain indecomposable elements of the fs monoids $\oM_{X,x}$. This does not apply directly to a log variety $X$ over $R$ because the monoid $\oM_{X,x}$ is too large. It is unclear if there is a canonical Zariski subdivision in this case, though one can construct a non-canonical one by blowing up large enough ideals. Instead of working this out, we will reduce to the fine case using the following approximation result.

\begin{theor}\label{approxth}
Assume that $\calO$ is a valuation ring, and let $\{R_i\}_{i\in I}$ denote the family of fine submonoids of $R=\calO\setminus\{0\}$. Let $S$ and $S_i$ denote $\Spec(\calO)$ provided with the log structures $R$ and $R_i$, respectively. Then for any log variety $X=(X,M_X)$ over $S$ there exists $i\in I$ and a morphism of fine log schemes $X_i=(X,M_i)\to S_i$ with an isomorphism $X=X_i\times_{S_i}S$, where the product is taken in the category of integral log schemes.
\end{theor}
\begin{proof}
Assume first that $X$ possesses an affine global chart $P\to\Gamma(\calO_X)$ with a finitely presented $R$-monoid $P$. By Lemma~\ref{approxlem}, $P=(P_j\oplus_{R_j}R)^\int$ for a large enough $j$ and an integral finitely generated $R_j$-monoid $P_j$. Therefore, we can take $i=j$ and $M_i$ the log structure associated with $R_i\to\Gamma(\calO_X)$.

In general, there exists a strict \'etale covering $\coprod_{k=1}^n X_k\to X$ such that each $X_k$ possesses a global affine chart. By the affine case, choosing $i\in I$ large enough we can provide each $X_k$ with a log structure $M_{k,i}$ such that $X_{k,i}=(X_k,M_{k,i})$ is a fine log scheme and $X_k=X_{k,i}\times_{S_i}S$. Moreover, by the second claim of Lemma~\ref{approxlem}, any two choices of $M_{k,i}$ become isomorphic after increasing $i$. Therefore, for a large enough $j\ge i$ the log structures of $X_{k,j}=X_{k,i}\times_{S_i}S_j$ agree on the products $X_{k_1}\times_{X}X_{k_2}$ and hence give rise to a required fine log structure $M_j$ on the whole $X_j$.
\end{proof}

\subsubsection{Global fans}\label{zarthproof}
Now we can extend the Zariski subdivision theorem to log varieties over $S$.

\begin{theor}\label{zarth}
Let $X$ be a log variety over $\calO$. Then there exists a monoidal subdivision $X'\to X$ such that $X'$ is a Zariski log variety that possesses a global fan $X'\to F$.
\end{theor}
\begin{proof}
Find an isomorphism $X=X_i\times_{S_i}S$ as in Theorem~\ref{approxth}. By Lemma~\ref{zarlem2} there exists a monoidal subdivision $X'_i\to X_i$ such that $X'_i$ is Zariski and possesses a fan $F_i$. Then $X'=X'_i\times_{X_i}X$ is a monoidal subdivision of $X$, and since $X'=X'_i\times_{S_i}S$, we also see that the log structure of $X'$ is Zariski and $F_i\times_{\oSpec(R_i)}\oSpec(R)$ is its global fan.
\end{proof}

\subsubsection{Monoidally polystable subdivisions}
A log variety $X$ over $S$ is called {\em monoidally polystable} at a point $x\in X$ if the $R$-toric monoid $\oM_{X,x}$ is polystable. If, in addition, $X$ is log smooth over $S$ at $x$ then we say that $X$ is {\em polystable} at $x$. Naturally, $X$ is {\em monoidally polystable} or {\em polystable} if it is so at all points.

\begin{rem}
By \S\ref{polymonsec}, $X$ is polystable at $x$ if and only if \'etale locally over $x$ it admits a strict \'etale morphism to a model polystable log variety of the form $$\Spec(R[u_1^{\pm 1}\..u_m^{\pm 1},t_{0,0}\..t_{0,n_0}\..t_{l,0}\..t_{l,n_l}]/(t_{1,0}\cdot\dots\cdot t_{1,n_1}-\pi_1\..t_{l,0}\cdot\dots\cdot t_{l,n_l}-\pi_l))$$
with the log structure generated by $t_{ij}$ over $R$. So, on the level of schemes our notion of polystability over $S$ agrees with the one introduced by Berkovich in \cite[Section~1]{bercontr}. However, our definition works, more generally, with log schemes and addresses the case of a non-trivial horizontal log structure corresponding to $t_{0,0}\..t_{0,n_0}$.
\end{rem}

We say that the log structure on $X$ is {\em vertical} if its restriction on the generic fiber $X_\eta$ is trivial.

\begin{theor}\label{monoidalth}
Let $\calO$ be a valuation ring of height 1 whose group of values is divisible and let $X$ be a log variety. Assume that the log structure of $X$ is vertical. Then there exists a monoidal subdivision $X'\to X$ such that $X'$ is monoidally polystable. In particular, if $X$ is log smooth over $S$, then $X'$ is polystable over $S$.
\end{theor}
\begin{proof}
First, by Theorem~\ref{zarth} there exists a monoidal subdivision $X''\to X$ such that $X''$ possesses a global fan $X''\to F$. In particular, $X''_\eta=X_\eta$ has trivial log structure and hence the fan $F$ corresponds to a polyhedral complex whose cells are bounded. Then by Theorem~\ref{compsubdv} there exists a subdivision $F'\to F$ such that $F'$ is polystable. Since $F'$ is a global fan of the monoidal subdivision $X'=X''\times_FF'$ of $X$, we obtain that $X'$ is monoidally polystable, as required.
\end{proof}

\begin{rem}
For simplicity we only considered polyhedral complexes with bounded cells in \S\ref{combsec}, but the theory can be extended to the unbounded case. As a corollary one would be able to drop the assumption that $X_\eta\subseteq X_{\mathrm{tr}}$ in Theorem~\ref{monoidalth}.
\end{rem}

\subsubsection{Polystable $p$-alteration theorem}
Finally, we combine our main result with the $p$-alteration theorem to obtain the following application, where as in \cite[\S4.1.2]{tame-distillation} an {\em alteration} means a proper, surjective, maximally dominating, generically finite morphism:

\begin{theor}\label{appth}
Assume that $\calO$ is a valuation ring of height 1 and residual characteristic exponent $p$. Then for any scheme $X$ flat and of finite presentation over $\calO$ there exist an extension of valuation rings $\calO\subseteq\calO'$ and a $p$-alteration $X'\to X\otimes_\calO\calO'$ such that the extension $K'/K$ of fields of fractions is finite and $X'$ provided with the log structure induced from the closed fiber $X'_s$ is polystable over $\calO'$.
\end{theor}
The condition on the log structure means that $M_{X'}=\calO_{X'}\cap i_*\calO^\times_{X'_\eta}$. In particular, the generic fiber $X'_\eta$ is a smooth $K'$-variety with the trivial log structure.
\begin{proof}
Let $\ocalO$ be an extension of $\calO$ to an algebraic closure $\oK$. If $X\otimes_\calO\ocalO$ admits such an alteration $\oX$, then the latter is induced from an alteration $X'\to X\otimes_\calO\calO'$ for a large enough finite extension $K'/K$ and $\calO'=\ocalO\cap K'$. Also, it is easy to see that after an additional increasing of $K'$ one achieves that $X'$ is polystable over $\calO'$. So, we can safely assume in the sequel that $K=\oK$.

Next, depending on the characteristics of $K$ and $\calO/\mathfrak{m}_\calO$, we present $\calO$ as the filtered union of subrings $\calO_i$, $i\in I$ finitely generated over $\bfQ$, $\bfZ_{(p)}$ or $\bfF_p$. In particular, $S=\Spec(\calO)$ is the filtered limit of affine schemes $S_i=\Spec(\calO_i)$. By approximation theory (see \cite[$\rm IV_2$, \S8]{ega}), for a large enough $i\in I$, that we now fix, there exists an $S_i$-scheme $X_i$ of finite type such that $X=X_i\times_{S_i}S$.

By \cite[Theorem~1.2.9]{tame-distillation} there exist $p$-alterations $X'_i\to X_i$ and $S'_i\to S_i$ and log structures on $S'_i$ and $X'_i$, induced by an appropriate divisor on $S'_i$ and its preimage in $X'_i$, such that $X'_i\to S'_i$ is log smooth. Since $K$ is assumed to be algebraically closed and $S'_i\to S_i$ is proper, the morphism $S\to S_i$ factors through $S'_i$. Furthermore, $\calM_{S'_i}\into\calO_{S'_i}\setminus\{0\}$, hence $S\to S'_i$ upgrades to a morphism of log schemes for the standard log structure $R=\calO\setminus\{0\}$ on $S$. Setting $X'=X'_i\times_{S'_i}S$ we obtain a log smooth log variety $X'$ over $\calO$, whose log structure is induced by the closed fiber $X'_s$. The morphisms $X'\to S$ and $X'\to X_i$ induce a morphism $X'\to X_i\times_{S_i}S=X$, which is easily seen to be a $p$-alteration too. Finally, by Theorem~\ref{monoidalth} we can further replace $X'$ by its modification (even a monoidal subdivision) so that it becomes polystable over $S$.
\end{proof}

\begin{rem}
In the case of residual characteristic zero, $X'\to X\otimes_\calO\calO'$ is a modification, and it suffices to use \cite[Theorem~2.1]{AK} instead of \cite[Theorem~1.2.9]{tame-distillation}.
\end{rem}

\bibliographystyle{myamsalpha}
\bibliography{log_smooth}

\end{document}